\documentclass[11pt]{article}

\usepackage{fullpage}
\usepackage{amssymb}
\usepackage{amsmath}
\usepackage{amsthm}
\usepackage{thmtools}
\usepackage[all]{xy}
\usepackage{graphicx}
\usepackage{mathrsfs}
\usepackage{hyperref}
\usepackage[capitalise]{cleveref}
\usepackage[dvipsnames]{xcolor}
\usepackage{verbatim}
\usepackage{booktabs}
\usepackage{longtable}
\usepackage{array}
\usepackage{adjustbox}
\usepackage[pagewise]{lineno} 

\newcommand{\repcell}[1]{
  \begin{minipage}[t]{\linewidth}
  \raggedright\footnotesize
  \(\begin{aligned}[t]
  #1
  \end{aligned}\)
  \end{minipage}
}

\newcommand{\graphcell}[2][@-0pc]{
  \begin{minipage}[t]{\linewidth}
  \centering
  \begin{adjustbox}{valign=t,max width=\linewidth,max height=.18\textheight}
  \(\xymatrix#1{#2}\)
  \end{adjustbox}
  \end{minipage}
}

\hypersetup{
  hidelinks,
  colorlinks  = true,    
  urlcolor    = blue,    
  linkcolor   = blue,    
  citecolor   = blue,      
  pdfauthor   = {Gabe Cunningham, Mark Mixer},
}

\newtheorem*{acomment}{\color{BrickRed}{Comment}}

\newtheorem{theorem}{Theorem}[section]
\newtheorem{lemma}[theorem]{Lemma}
\newtheorem{proposition}[theorem]{Proposition}
\newtheorem{corollary}[theorem]{Corollary}

\theoremstyle{definition}

\newtheorem{hypothesis}[theorem]{Hypothesis}

\newcommand{\G}{\Gamma}

\newcommand{\mix}{\mathbin{\diamondsuit}}
\newcommand{\comix}{\mathbin{\square}}

\title{Representing alternating groups as self-dual string C-groups of high rank}

\author{
  Gabe Cunningham\thanks{Wentworth Institute of Technology, Boston, MA, USA 02115. Email: cunninghamg1@wit.edu (Corresponding Author)}
  \and
  Mark Mixer\thanks{Wentworth Institute of Technology, Boston, MA, USA 02115. Email: mixerm@wit.edu}
}

\date{ \today }

\begin{document}

\maketitle

\begin{abstract}

The highest rank of a string C-group representation of the alternating group $A_n$ is known for each $n$, but no self-dual representations attaining this highest rank are known when $n > 12$.
Motivated by computational results for alternating groups of small degree, we examine a vertex-gluing construction for permutation representation graphs.  We establish conditions under which gluing two string C-groups produces another string C-group, and use this construction to obtain infinite families of self-dual representations of alternating groups.  In particular, for every $n = 4m+3 \geq 15$, we construct $\left \lfloor \frac{n+9}{8} \right \rfloor$ 
distinct self-dual string C-groups of rank $2m$ isomorphic to $A_{n}$.  These representations have rank one below the maximum possible rank of string C-group representations for $A_n$, and to the authors' knowledge are the highest-rank self-dual representations currently known for
alternating groups.


\noindent MSC 20B25, 05E18, 20D06, 52B11

\medskip
\noindent

\end{abstract}

\section{Introduction}

\emph{String C-groups} of rank $r$ are quotients of the universal string Coxeter group
\[ W_r = \langle \rho_0, \ldots, \rho_{r-1} \mid \rho_i^2 = 1 ~\forall  i, \ (\rho_i \rho_j)^2 = 1 ~\forall i,j \textrm{ such that } |i-j| \geq 2 \rangle \]
that inherit its \emph{intersection condition}
\begin{equation}
\langle \rho_i \mid i \in I \rangle \cap \langle \rho_j \mid j \in J \rangle = \langle \rho_k \mid k \in I \cap J \rangle.
\end{equation}
The study of string C-group representations of finite groups is a rich and varied field, particularly when it comes to alternating and symmetric groups. Of particular interest is the question: Given a group $\G$, what is the highest rank $r$ such that there is a string C-group $\Lambda = \langle \rho_0, \ldots, \rho_{r-1} \rangle$ with $\Lambda \cong \G$? In a series of papers, it was established that for each $n$, the highest rank $r$ of a string C-group representation of the alternating group $A_n$ is $\lfloor \frac{n-1}{2} \rfloor$ when $n \geq 12$ \cite{CFLM2016,highest-rank-an,FLM12,FLM12A}.

The string C-group $\G = \langle \rho_0, \ldots, \rho_{r-1} \rangle$ is said to be \emph{self-dual} if there is a group automorphism of $\G$ that sends each $\rho_i$ to $\rho_{r-i-1}$. The string C-groups of highest rank for $A_n$ that are described in \cite{FLM12} are not self-dual. So it is natural to ask: What is the highest rank of a self-dual string C-group for $A_n$? Using Magma \cite{Magma}, we computed this highest rank for $n \leq 19$. For comparison, in \cref{tab:self-dual-an} we include the highest rank of all string C-group representations of $A_n$ when we do not restrict ourselves to self-dual string C-groups. 

For some small $n$-values, the highest rank is unchanged -- in other words, some of the highest rank string C-groups for $A_n$ are already self-dual. In general, however, the data suggest that for $n \geq 13$, the largest rank of a self-dual string C-group for $A_n$ is at most $\lfloor \frac{n-1}{2} \rfloor - 1$, one less than the highest rank for alternating groups in general. 

\begin{table}
\begin{center}
\begin{tabular}{|c|c|c|c|c|} \hline
$n$ & Max Rank SD & Max Rank & \# of Examples & Reference \\ \hline
$\leq 4$ & N/A & N/A & N/A & \\ \hline
5 & 3 & 3 & 1 & $\{5,5\}_3$ \\ \hline
6-8 & N/A & N/A & N/A &  \\ \hline
9 & 4 & 4 & 1 & \cite[Figure 4]{FLM12A}, type $\{5,6,5\}$ \\ \hline
10 & 3 & 5 & 11 &  Appendix~\ref{appendix:A10} \\ \hline
11 & 6 & 6 & 2& \cite[Figure 6]{FLM12A}, \\ &&&& types $\{5,3,6,3,5\}$, $\{5,5,6,5,5\}$ \\ \hline
12 & 5 & 5 & 1 &  Appendix~\ref{appendix:A12} \\ \hline
13 & 4 & 6 & 2 &  Appendix~\ref{appendix:A13}\\ \hline
14 & 4 & 6 & 1 &   Appendix~\ref{appendix:A14} \\ \hline
15 & 6 & 7 & 3 & Corollary \ref{cor:3examples}\\ \hline
16 & 6 & 7 & 3 &   Appendix~\ref{appendix:A16}\\ \hline
17 & 6 & 8 & 2 &   Appendix~\ref{appendix:A17} \\ \hline
18 & 4 & 8 & 4 &  Appendix~\ref{appendix:A18} \\ \hline
19 & 8 & 9 & 3 & Corollary \ref{cor:3examples} \\ \hline
\end{tabular}

\caption{Comparing the highest rank of a self-dual string C-group to the overall highest rank of a string C-group isomorphic to the alternating group $A_n$.  The first column gives the degree $n$.  The second column gives the highest rank of a self-dual string C-group for $A_n$.  The third column gives the highest rank of a string C-group for $A_n$.  The fourth column gives the number of self-dual examples of highest rank. \label{tab:self-dual-an}}
\end{center}
\end{table}

When $n = 15$ or $n = 19$, we found exactly three self-dual string C-groups of rank $\lfloor \frac{n-1}{2} \rfloor - 1$ and none of higher rank. Furthermore, the structure of these groups is special; their permutation representation graphs can be naturally split into two pieces $H_1$ and $H_2$ (with one shared vertex) such that $H_2$ is the dual of $H_1$. We verified computationally that similar graphs worked to produce self-dual string C-groups of rank $\lfloor \frac{n-1}{2} \rfloor - 1$ for $n = 23, 27$, and $31$. Furthermore, as $n$ increased, we found more self-dual string C-groups of this rank, each of which also exhibited a similar structure. In this paper, we focus on those self-dual string C-groups whose permutation representation graphs can be split in this way.  In~\cite{leemans-mulpas-gluing}, similar graphs for symmetric groups are studied.  We focus on even groups, and are able to prove the following.

\begin{theorem}
\label{thm:real-main}
For each $n = 4m+3$, with $n \geq 15$, there are at least $\lfloor \frac{n+9}{8} \rfloor$ distinct self-dual string C-groups of rank $2m$ for $A_{n}$.
\end{theorem}

The paper is organized as follows.  In \cref{sec:background}, we review the necessary background information on string C-groups, permutation groups, mixes and comixes, and orientability.   In \cref{sec:glue}, we introduce a vertex-gluing construction and establish conditions under which gluing two string C-groups produces another string C-group.  In \cref{sec:families}, we apply these results and obtain nine families of string C-groups for alternating groups, including three infinite families of self-dual examples.  In \cref{sec:moreboxes}, we generalize one family to obtain the collection of 
$\left\lfloor \frac{n+9}{8}\right\rfloor$ 
non-isomorphic self-dual string C-groups of rank $2m$ for $A_{n}$ when $n=4m+3 \geq 15$.
We conclude with open problems and a conjecture concerning the highest rank of self-dual string C-group representations of alternating groups.

\section{Background} \label{sec:background}

\subsection{String C-groups}

The \emph{universal string Coxeter group of rank $r$} is the group
\[ W_r = \langle \rho_0, \ldots, \rho_{r-1} \mid \rho_i^2 = 1 ~\forall  i, \ (\rho_i \rho_j)^2 = 1 ~\forall i,j \textrm{ such that } |i-j| \geq 2 \rangle \]
A \emph{rank $r$ string group generated by involutions (sggi)} is a quotient of $W_r$; in other words, a group $\G = \langle \rho_0, \ldots, \rho_{r-1} \rangle$ satisfying $\rho_i^2 = 1$ for all $i$ and $(\rho_i \rho_j)^2 = 1$ whenever $|i - j| \geq 2$. Note that some of the generators of $\G$ may be trivial or may coincide.

Whenever $\G = \langle \rho_0, \ldots, \rho_{r-1} \rangle$ is an sggi, we define:
\[ \G_{r-1} = \langle \rho_0, \ldots, \rho_{r-2} \rangle, \]
\[ \G_0 = \langle \rho_1, \ldots, \rho_{r-1} \rangle, \]
\[ \G_{0,r-1} = \langle \rho_1, \ldots, \rho_{r-2} \rangle. \]

Of particular interest are \emph{string C-groups}, which are sggis that satisfy the following \emph{intersection condition}:
\begin{equation}
\langle \rho_i \mid i \in I \rangle \cap \langle \rho_j \mid j \in J \rangle = \langle \rho_k \mid k \in I \cap J \rangle.
\end{equation}
Each string C-group $\G$ (with this distinguished set of generators) uniquely determines an abstract regular polytope $\mathcal{P}$ whose automorphism group is $\G$ (see \cite[Section 2E]{ARP}). Thus, questions about abstract regular polytopes can be entirely expressed in the language of string C-groups.

In this paper, we will frequently build groups that will obviously be sggis, and our goal is to show that they are string C-groups. We will also simply say that they are \emph{string C}, as an adjective. Here we present two helpful results that greatly reduce the number of cases of the intersection condition that we need to check.

\begin{proposition}\label{int-cond}
Suppose $\Gamma = \langle \rho_0, \ldots, \rho_{r-1} \rangle$ is an sggi, with no $\rho_k$ trivial and with $\rho_k \neq \rho_t$ whenever $k \neq t$. If $\langle \rho_0, \ldots, \rho_j \rangle \cap \langle \rho_i, \ldots, \rho_{r-1} \rangle = \langle \rho_i, \ldots, \rho_j \rangle$ for all $1 \leq i \leq j \leq r-2$, then $\Gamma$ is a string C-group.
\end{proposition}
\begin{proof}
We proceed by induction on $r$.  First, suppose $r = 2$. Every group $\langle \rho_0, \rho_1 \rangle$ with no $\rho_k$ trivial and $\rho_0 \neq \rho_1$ is a string C-group.

Now, suppose the result is true for sggis of rank $r-1$, and consider $\Gamma_{r-1} = \langle \rho_0, \ldots, \rho_{r-2} \rangle$. Note that $\Gamma_{r-1}$ also has no $\rho_k$ trivial, and $\rho_k \neq \rho_t$ when $k \neq t$. Furthermore, for each $1 \leq i \leq j \leq r-3$:

\[ \langle \rho_0, \ldots, \rho_j \rangle \cap \langle \rho_i, \ldots, \rho_{r-2} \rangle \leq \langle \rho_0, \ldots, \rho_j \rangle \cap \langle \rho_i, \ldots, \rho_{r-1} \rangle = \langle \rho_i, \ldots, \rho_j \rangle.\] 

On the other hand, the right side is clearly contained in the left side, so $\langle \rho_0, \ldots, \rho_j \rangle \cap \langle \rho_i, \ldots, \rho_{r-2} \rangle = \langle \rho_i, \ldots, \rho_j \rangle$. Thus, by the inductive hypothesis, $\Gamma_{r-1}$ is a string C-group. A similar argument establishes that $\Gamma_0$ is a string C-group. Then $\Gamma_{r-1} \cap \Gamma_0 = \langle \rho_0, \ldots, \rho_{r-2} \rangle \cap \langle \rho_1, \ldots, \rho_{r-1} \rangle$, which is equal to $\langle \rho_1, \ldots, \rho_{r-2} \rangle$ by assumption. Then by \cite[Prop. 2E16]{ARP}, $\Gamma$ is a string C-group.
\end{proof}

\begin{lemma}[Lemma 2.2 of \cite{FLM12A}] \label{max}
Assume $\G_0$ and $\G_{r-1}$ are string C-groups. If $\rho_{0} \not \in \G_{0}$ and if $\G_{0,r-1}$ is maximal in $\G_{r-1}$, then $\G$ is string C.
\end{lemma}

The \emph{dual} of the sggi $\G$ is the group $\G^* = \langle \rho_{r-1}, \ldots, \rho_0 \rangle$. This is clearly an sggi, and will be string C if and only if $\G$ is. We say that $\G$ is \emph{self-dual} if there is an isomorphism from $\G$ to itself that sends each $\rho_i$ to $\rho_{r-1-i}$. 

The \emph{type} or \emph{Schl\"afli symbol} of a rank $r$ sggi $\langle \rho_0, \ldots, \rho_{n-1} \rangle$ is a symbol
\[ \{k_1, \ldots, k_{r-1} \}, \]
where $k_i = |\rho_{i-1} \rho_i|$. Whenever there are repeated elements in a type, we sometimes abbreviate the type using exponents, so that $\{4, 3^2, 4\}$ should be understood as the type $\{4, 3, 3, 4\}$.

\subsection{Permutation groups}

Let  $\G=\langle \rho_0,\ldots,\rho_{r-1}\rangle
   \leq S_{\Omega}$
be an sggi represented as a permutation group on a finite set $\Omega$.
In this paper we will often be interested in even permutation groups, and we define $\G^+ = \G \cap A_{\Omega}.$. When the context is clear, we will use $2$ to represent the unique group of order 2.   For instance, $(2 \wr S_{m})^+$ represents the even subgroup of the wreath product of the group of order 2 and a symmetric group $S_m$.

The \emph{permutation representation graph} of $\G$ is the
edge-labeled (multi)graph whose vertex set is $\Omega$ and in which vertices
$x$ and $y$ are joined by an edge labeled $i$ whenever
$(x,y)$ is a transposition in the disjoint cycle decomposition of
$\rho_i$.  Equivalently, there is an $i$-edge between $x$ and $y$ if
and only if $
x\rho_i=y$ for $y \neq x$.
A point fixed by $\rho_i$ is not incident to any edge labeled $i$.

Since each $\rho_i$ is an involution, the edges of any fixed label
form a matching, where every vertex is incident to at most one edge of that
label.  Edges with different labels may join the same pair
of vertices, and thus permutation representation graphs can contain multiple (parallel) edges between the same two vertices.  Such a labeled graph determines a permutation
representation, since each generator $\rho_i$ can be recovered by
interchanging the endpoints of every edge labeled $i$ and fixing all
remaining vertices.  When we refer to the group of a permutation
representation graph, we mean the group generated by these
involutions.

The connected components of the permutation representation graph are
precisely the orbits of $\G$ on $\Omega$.  Additionally, the orbits of $
\G_0 $ are the connected components of the graph obtained by deleting every
edge whose label is $0$.  In general, the permutation representation
graphs of subgroups generated by subsets of the distinguished
generators can be read directly from the original graph.

The permutation representation graph also records the parity of each generator.  In particular,
$\rho_i$ is even if and only if the graph contains an even number of
edges labeled $i$.  A group $\G$ is an even permutation group
if and only if every distinguished generator is even.

Finally, the permutation representation graph of the dual group is obtained by replacing every edge label $i$ by $r-1-i$.  In particular, a symmetry of a permutation representation graph that reverses the edge labels induces a self-duality of the corresponding sggi.

We next present several results about permutation groups that will be useful in later sections.

The following proposition follows from~\cite[Theorem 3.3E]{DM96}.
\begin{proposition}\label{prime-cycle}
Let $\G$ be a primitive subgroup of $S_n$.  If $\G$ contains a prime cycle that fixes at least 3 points, then $\G$ is either $A_n$ or $S_n$.
\end{proposition}

\begin{lemma} [Lemma 5.1 of \cite{CPR}]  \label{sym}
If a subgroup $\Gamma$ of $S_n$ contains the transposition $(n - 1, n)$ as well as a subgroup
acting transitively on $\{1,\ldots, n-1\}$ while keeping $n$ fixed, then $\Gamma = S_n$.
\end{lemma}

Note that we do not require that $G$ can be embedded in a wreath product, as the transitivity of $G$ is not required, nor is it required that the blocks all have equal sizes.

\subsection{Mix and Comix}

Whenever we have a group whose permutation representation graph is disconnected, the group decomposes using the idea of the \emph{parallel product} or \emph{mix} of two groups (see \cite{parallelProduct}, \cite[Section 7A]{ARP}, \cite{var-gps}) in a way analogous to the graph's own decomposition into connected components. We describe the basic theory below.

Consider groups $\G = \langle x_1, \ldots, x_n \mid R \rangle$ and $\Lambda = \langle y_1, \ldots, y_n \mid S \rangle$, with the possibility that some $x_i$ or $y_i$ are trivial. Then the \emph{mix} of $\G$ and $\Lambda$ is the group 
\[ \G \mix \Lambda = \langle (x_1, y_1), \ldots, (x_n, y_n) \rangle \leq \G \times \Lambda. \]
It is clear that this is a subdirect product (that is, it projects onto $\G$ and $\Lambda$). Furthermore, a word $(x_{i_1}, y_{i_1}) \cdots (x_{i_n}, y_{i_n})$ is trivial if and only if $x_{i_1} \cdots x_{i_n}$ and $y_{i_1} \cdots y_{i_n}$ are both trivial. This means that if we write $\G$ and $\Lambda$ as quotients of the universal string Coxeter group, say $W_r / M_\G$ and $W_r / M_\Lambda$, then $\G \mix \Lambda \cong W_r / (M_\G \cap M_\Lambda)$. Note that the mix of two groups is naturally commutative and associative. 

Now, suppose that $\G$ and $\Lambda$ are rank $r$ sggis with permutation representation graphs $G_\G$ and $G_\Lambda$, respectively. If we consider the disjoint union of $G_\G$ and $G_\Lambda$, this defines a new rank $r$ sggi that projects onto $\G$ and $\Lambda$. Now, we can identify $\G$ with the group $W_r / K_\G$, where $K_\G$ is the kernel of the action of $W_r$ on $G_\G$. We can identify $\Lambda$ with $W_r / K_\Lambda$ defined similarly. Then the kernel of the action of $W_r$ on $G_\G \sqcup G_\Lambda$ is $K_\G \cap K_\Lambda$. Thus we see that the group with permutation representation graph $G_\G \sqcup G_\Lambda$ is precisely the mix $\G \mix \Lambda$.

Note that if $\G$ and $\Lambda$ are permutation groups acting on sets of equal size in exactly the same way, then $\G \mix \Lambda \cong \G$. In other words, if a permutation group has a permutation representation graph that has two or more copies of any connected component, then we may as well throw away the duplicates.

In order to further determine the structure and size of $\G \mix \Lambda$, it is useful to consider the \emph{comix} $\G \comix \Lambda$. This is the group
\[ \G \comix \Lambda = \langle x_1, \ldots, x_n \mid R, \overline{S} \rangle, \]
where $\overline{S}$ consists of the words in $S$ but with each $y_i$ replaced with $x_i$. In other words, $\G \comix \Lambda$ is the largest quotient of $\G$ and $\Lambda$ that identifies each $x_i$ with $y_i$. This implies that $\G \comix \Lambda \cong W_r/(M_\G M_\Lambda)$.
 The index of $\G \mix \Lambda$ in $\G \times \Lambda$ is equal to the size of $\G \comix \Lambda$. In particular, $\G \mix \Lambda \cong \G \times \Lambda$ if and only if $\G \comix \Lambda$ is trivial. 

Note that the kernel of the natural projection from $\G \mix \Lambda \to \Lambda$ is $M_\Lambda / (M_\G \cap M_\Lambda)$, while the kernel of the natural projection from $\G \to \G \comix \Lambda$ is $M_\G M_\Lambda / M_\G$. These are isomorphic as groups by the Second Isomorphism Theorem. 

One of our primary applications in this paper involves analyzing a permutation group that acts as a symmetric group or alternating group on one orbit. We will use the following result to help us determine the structure of such a group.

\begin{lemma}
\label{mix-with-sn}
Let $\G$ be an even permutation group acting on disjoint sets $\Omega$ and $\Delta$. Let $\G_\Omega$ and $\G_\Delta$ be the permutation groups of the induced actions on the respective sets. Suppose there is some $\alpha \neq 1 \in \G$ such that $\alpha$ acts trivially on $\Omega$, and suppose that $\G_\Delta$ is $A_n$ or $S_n$. If $n \neq 4$, or if $\alpha$ acts as a $3$-cycle on $\Delta$, then 
$\G$ contains all pairs $(\pi, \beta)$ in $\G_\Omega \times \G_\Delta$ such that $\pi$ and $\beta$ have the same parity.
\end{lemma}

\begin{proof}
The permutation group $\G$ is the mix $\G_\Omega \mix \G_\Delta$. The fact that some nontrivial $\alpha \in \G$ acts trivially on $\Omega$ means that the mix contains an element $(1, \alpha_2)$, and so the kernel $K$ of the projection from $\G_\Omega \mix \G_\Delta \to \G_\Omega$ is nontrivial. This kernel is isomorphic to the kernel of the projection from $\G_\Delta$ to $\G_\Omega \comix \G_\Delta$, and so this is nontrivial as well; in other words, $\G_\Omega \comix \G_\Delta$ is a proper quotient of $\G_\Delta$. If $n = 4$ and $\alpha$ acts as a 3-cycle on $\Delta$, then the kernel from $\G_\Delta$ to $\G_\Omega \comix \G_\Delta$ contains a 3-cycle, and then must contain $A_4$. Thus, no matter what $n$ is, $K$ contains $A_n$. 
This implies that $\G_\Omega \mix \G_\Delta$ contains every element $(1, \beta)$ with $\beta \in A_n \leq \G_\Delta$. 

Now, let $\varphi_\Omega \in \G_\Omega$ and let $\gamma \in \G_\Delta$ have the same parity as $\varphi_\Omega$.  Then there is some $\varphi = (\varphi_\Omega, \varphi_\Delta) \in \G = \G_\Omega \mix \G_\Delta$, where $\varphi_A$ is the induced action of $\varphi$ on $A$ for $A \in \{\Omega, \Delta\}$. Since $\G$ is even, $\varphi_\Omega$ and $\varphi_\Delta$ have the same parity. Then $\varphi_\Delta^{-1} \gamma$ is even, so multiplying $(\varphi_\Omega, \varphi_\Delta)$ by $(1, \varphi_\Delta^{-1} \gamma)$ produces the desired element $(\varphi_\Omega, \gamma)$.

\end{proof}

\subsection{$k$-orientability}

For each $0 \leq k \leq r-1$, the rank $r$ Coxeter Group $W_r = \langle \rho_0, \ldots, \rho_{r-1} \rangle$ has an index 2 subgroup
\[ W_r^{(k)}
=
\left\langle
\rho_j,\ \rho_k\rho_j\rho_k
\mid j\neq k
\right\rangle.
\]
Note that when $|j-k| > 1$ we have $\rho_k \rho_j \rho_k = \rho_j$. 

If $\G = W_r/M$ is an sggi, then we may similarly define the subgroup $\G^{(k)}$. The index of this subgroup in $\G$ is at most 2, and if the index is 2, then we will say that $\G$ is \emph{$k$-orientable}. This happens precisely when $M$ is contained in $W_r^{(k)}$. For more information on $k$-orientability and similar ideas, see \cite{flag-bicolorings}.

The idea of $k$-orientability is related to the simplest example of mixing, where we take a permutation representation graph for $\G$ and add a single new transposition. (This is essentially the same idea as a \emph{sesqui-extension} in \cite{FLM12}.) We will use $E_k$ to denote the permutation group $\langle \tau_0, \ldots, \tau_{r-1} \rangle$ (with the rank $r$ implicit from the context) such that $\tau_k$ is a transposition and every other $\tau_i$ is trivial. (We will usually pick $\tau_k$ to interchange two elements not moved by $\G$.) Then there are two possibilities for $\G \mix E_k$; either $\G \mix E_k \cong \G$ (which happens if and only if $\G$ covers $E_k$, so that $\G \comix E_k = E_k$) or $\G \mix E_k \cong \G \times E_k$ (which happens if and only if $\G \comix E_k$ is trivial). We then have the following result. (Compare to \cite[Lemma 5.4]{FLM12}.)

\begin{proposition}
\label{k-orientable}
Let $\G$ be an sggi with no trivial generators. Then the following are equivalent:
\begin{enumerate}
\item[(1)] $\G$ is $k$-orientable.
\item[(2)] $\rho_k \not \in \G^{(k)}$.
\item[(3)] $\G \mix E_k \cong \G$.
\end{enumerate}
\end{proposition}

\begin{proof}
Let $\G = \langle \rho_0, \ldots, \rho_{r-1} \rangle$. It is clear that $\G^{(k)} = \G$ if and only if $\rho_k \in \G^{(k)}$, which shows that (1) and (2) are equivalent.

Now, suppose $\G$ is $k$-orientable. Then $\G^{(k)}$ has index 2 in $\G$, and the quotient is precisely $E_k$. It follows that $\G \comix E_k \cong E_k$, and so $\G \mix E_k \cong \G$. Conversely, if $\G$ is not $k$-orientable, then $\G \comix E_k \not \cong E_k$. Then $\G \comix E_k$ must be trivial, which implies $\G \mix E_k \cong \G \times E_k$. Thus (1) and (3) are equivalent. 
\end{proof}

\begin{corollary}
\label{k-mix}
$\G \mix E_k$ is $k$-orientable.
\end{corollary}

\begin{proof}
This follows from \cref{k-orientable} since $(\G \mix E_k) \mix E_k \cong \G \mix (E_k \mix E_k) \cong \G \mix E_k$.
\end{proof}

Many of our results will require certain permutation groups to be even. Here we see that if there is a single odd generator, then we can mix with an edge to make the group even without changing the (abstract) group.

\begin{proposition}
\label{mix-is-iso}
Suppose $\G$ is a permutation group such that $\rho_k$ is the only generator that is odd. Then $\G \mix E_k \cong \G$ and $\G$ is $k$-orientable.
\end{proposition}

\begin{proof}
By \cref{k-orientable}, it suffices to show that $\rho_k \not \in \G^{(k)}$. The generators of $\G^{(k)}$ are all $\rho_j$ and $\rho_k \rho_j \rho_k$ with $j \neq k$, and by assumption, these are all even permutations. Thus, since $\rho_k$ is odd, $\rho_k \not \in \G^{(k)}$ and the result follows.
\end{proof}

Next we describe a simple method for demonstrating that a permutation group is $k$-orientable.

\begin{proposition}
\label{k-coloring}
Suppose that $\G$ has a permutation representation graph whose vertices can be colored black and white such that $\rho_k$ does not fix any vertices and connects black vertices to white ones, while every other $\rho_j$ connects vertices of the same color. Then $\G$ is $k$-orientable.
\end{proposition}

\begin{proof}
Given such a coloring, we see that $\G^{(k)}$ preserves the color of every vertex. Since $\rho_k$ reverses the color, $\rho_k \not \in \G^{(k)}$ and thus $\G$ is $k$-orientable by \cref{k-orientable}.
\end{proof}

We often want to know when $\G \mix E_k$ is a string C-group. We start with the following well-known result.

\begin{lemma}[Proposition 5.3 of \cite{FLM12}] \label{mix-edge}
If $\G$ is a string C-group of rank $r$, then $\G \mix E_0$ and $\G \mix E_{r-1}$ are string C-groups.
\end{lemma}

We can also give some conditions for when $\G \mix E_k$ is a string C-group for other values of $k$.

\begin{proposition}
\label{mix-ek-string-c}
Suppose that $\G$ is a string C-group of rank $r$ and let $1 \leq k \leq r-2$. Then the mix $\G \mix E_k$ is a string C-group if any of the following are true:
\begin{enumerate}
\item $\G$ is $k$-orientable.
\item $\G_{r-1}$ is $k$-orientable.
\item $\G_0$ is $(k-1)$-orientable (after decreasing indices by 1).
\end{enumerate}
\end{proposition}

\begin{proof}
This follows directly from \cite[Proposition 6.13]{poly-mix}, but let us give a more self-contained proof here. If $\G$ is $k$-orientable, then $\G \mix E_k \cong \G$, which is string C. Otherwise, $\G \mix E_k \cong \G \times E_k$. The subgroup $(\G \mix E_k)_{r-1}$ is isomorphic to $\G_{r-1} \mix E_k$, and if $\G_{r-1}$ is $k$-orientable, then this is isomorphic to $\G_{r-1}$. Then the natural projection from $\G \mix E_k$ to $\G$ is one-to-one when restricted to $(\G \mix E_k)_{r-1}$, and so by \cite[Proposition 2E17]{ARP}, $\G \mix E_k$ is string-C. The last point follows from a dual argument.
\end{proof}

\begin{lemma}[Lemma 6.3 of \cite{FLM12}] 
\label{wreath}
For each $r \geq 3$, the group of the following permutation representation graph is a string C-group isomorphic to $2 \wr S_{r+1}$.

\[
\xymatrix@C=1.8pc@R=2.5pc{
*+[o][F]{} \ar@{-}[r]^0 &
*+[o][F]{} \ar@{-}[r]^1 &
*+[o][F]{} \ar@{-}[r]^2 &
*+[o][F]{} \ar@{-}[r]^3 &
*+[o][F]{} \ar@{.}[r] &
*+[o][F]{} \ar@{-}[r]^{r-2} &
*+[o][F]{} \ar@{-}[r]^{r-1} &
*+[o][F]{} \ar@{-}[d]^{r-2}
\\
*+[o][F]{} \ar@{-}[r]_0 &
*+[o][F]{} \ar@{-}[r]_1 &
*+[o][F]{} \ar@{-}[r]_2 &
*+[o][F]{} \ar@{-}[r]_3 &
*+[o][F]{} \ar@{.}[r] &
*+[o][F]{} \ar@{-}[r]_{r-2} &
*+[o][F]{} \ar@{-}[r]_{r-1} &
*+[o][F]{}
}
\]

\end{lemma}

\section{A vertex gluing operation} \label{sec:glue}
\subsection{Description of the operation} \label{sec:glue-op}
Suppose that $\G = \langle \rho_0 , \ldots \rho_{r_1-1} \rangle $ is a rank $r_1$ sggi, represented as a permutation group of degree $n_1$ and that $G_1$ is the permutation representation graph for this representation of $\G$.  
Similarly suppose that $\Lambda = \langle \alpha_0 , \ldots \alpha_{r_2-1} \rangle $ is a rank $r_2$ sggi, represented as a permutation group of degree $n_2$ and that $G_2$ is the permutation representation graph for this representation of $\Lambda$.  

If vertex $n_1$ in $G_1$ has degree 1 and is incident only to an edge of label $r_1-1$, and vertex $1$ in $G_2$ has degree 1 and is incident only to an edge of label $0$, then we can construct a rank $r_1+r_2$ sggi $\langle \phi_0, \ldots \phi_{r_1+r_2-1} \rangle $ with degree $n_1+n_2-1$ from $G_1$ and $G_2$ by ``gluing" vertices $n_1$ from $G_1$ and $1$ from $G_2$ as follows.

If $0 \le i < r_1$ then $\phi_i = \rho_i$.  Otherwise, when $r_1 \leq i < r_1+r_2$, we define $(n_1+x-1 ) \phi_i = (n_1+y-1)$ if and only if $x \alpha_{i-r_1}=y$.
This shifts the action of each $\alpha_j$ so that the point $1$ in $G_2$ becomes identified with point $n_1$ in $G_1$, and all other points from $G_2$ are reindexed accordingly.  We will say that two permutation groups $\G$ and $\Lambda$ can be glued together in this way, or equivalently that two permutation representation graphs $G_1$ and $G_2$ can be glued together.

For example, Figure~\ref{glue-ex} shows how this operation works by gluing the automorphism group of the 3-cube (acting on its facets) to the automorphism group of a 4-simplex (acting on its vertices).  The base facet of the cube is labeled by the point $n_1$, and the base vertex of the simplex is labeled by the point $1$.  The permutation representation graphs of both the cube and the simplex are shown above the new glued permutation representation graph.

\begin{figure}[h]
$$\includegraphics[width=10cm]{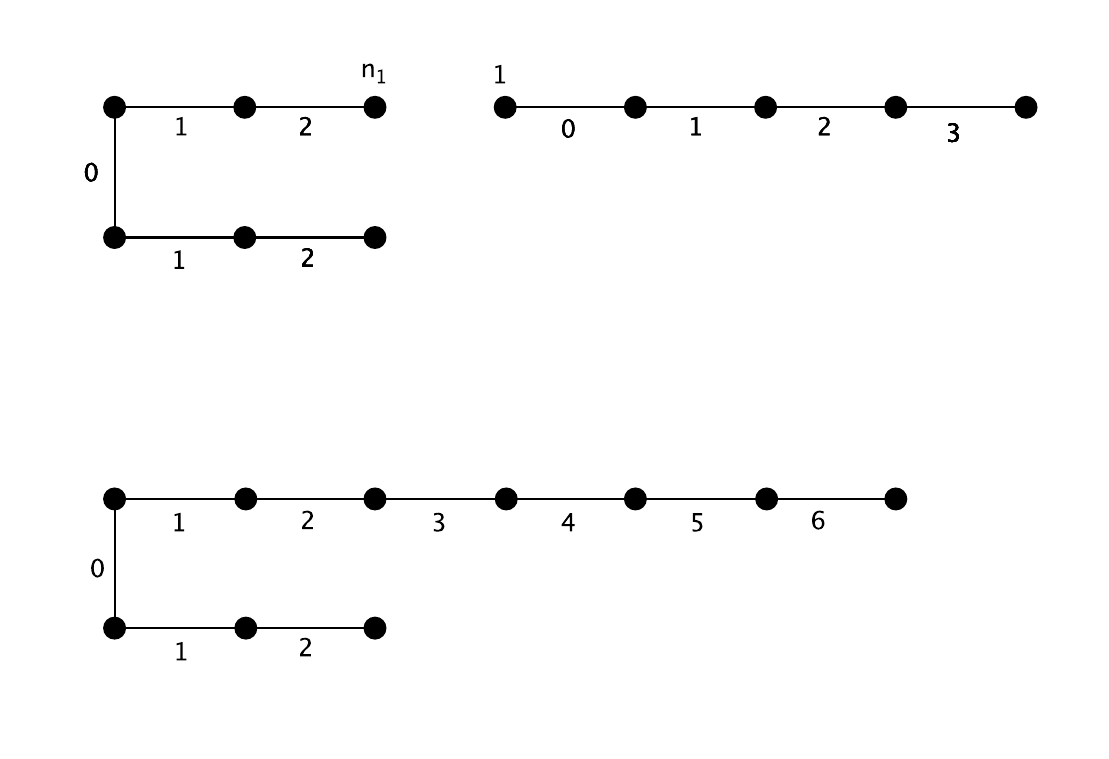}$$
\caption{A new sggi formed from a 3-cube and a 4-simplex \label{glue-ex}}
\end{figure}

Assume that as above $\G = \langle \rho_0 , \ldots \rho_{r_1-1} \rangle $ is a rank $r_1$ sggi, represented as a permutation group of degree $n_1$, and vertex $n_1$ in $G_1$ has degree 1 and is incident only to an edge of label $r_1-1$.  Then the dual of $\G$ will be an sggi that can be ``glued" to $\G $ as just described (after relabeling the points appropriately).  The result will be a self-dual sggi of rank $r=2r_1$ for a permutation group of degree $n=2n_1-1$.

\subsection{When is the result a string C-group?}
Our goal now is to consider sggis that can be decomposed into two pieces, and determine whether they are string C-groups. We will not initially assume that these sggis are the result of the gluing operation described above, but we will immediately prove that this indeed must be the case. We formalize the conditions we are considering in \cref{glued-pr} and  \cref{glued-cpr}.

\begin{hypothesis}
\label{glued-pr}
$\G$ is an sggi of rank $r$ with a permutation representation graph $G$ that consists of 
subgraphs $H_1$ and $H_2$ with at least two vertices each, sharing a single vertex $v$, where the edge labels in $H_1$ are between $0$ and $m-1$, and the edge labels in $H_2$ are between $m$ and $r-1$ (with $m < r$). Furthermore, $v$ is adjacent to at least one other vertex in $H_1$ and at least one other vertex in $H_2$.
\end{hypothesis}

\begin{proposition}
\label{v-edges}
If $\G$ satisfies \cref{glued-pr}, then $v$ is incident to exactly two edges, one of label $m-1$ and one of label $m$.
\end{proposition}

\begin{proof}
By assumption, $v$ is incident to a vertex $x \in H_1$ with an edge of label $i \leq m-1$ and to a vertex $y \in H_2$ with an edge of label $j \geq m$. Note that $\rho_i \rho_j$ sends $x$ to $y$ and $y$ to $v$, so $|\rho_i \rho_j| > 2$. This implies that $j = i+1$, so $i = m-1$ and $j = m$. Thus, these are the only edge labels possible for edges incident to $v$.
\end{proof}

By \cref{v-edges}, it follows that any group $\G$ satisfying \cref{glued-pr} will be the result of the operation described in \cref{sec:glue-op}.

\begin{proposition} \label{glued-primitive}
Suppose $\G = \langle \rho_0, \ldots, \rho_{r-1} \rangle$ satisfies \cref{glued-pr}. Then for all $0 \leq i \leq m-1$ and $m \leq j \leq r-1$, the restriction of $\langle \rho_i, \ldots, \rho_j \rangle$ to the orbit of $v$ is either $A_n$ or $S_n$, where $n$ is the number of points in that orbit. In particular, if $n \leq 5$, then the restriction is $S_n$.
\end{proposition}

\begin{proof}
Let $O$ be the orbit of $v$ under $\langle \rho_i, \ldots, \rho_j \rangle$, and let $\G' = \langle \rho_i', \ldots, \rho_j' \rangle$ be the restriction to $O$. Here some of the $\rho_k'$ may be trivial.  We will first show that $\G'$ is primitive. Suppose that $v$ is part of a nontrivial block $B$ under the action of $\G'$, and let $x \in B$ with $x \neq v$. Without loss of generality, assume that $x \in H_1 \setminus v$. Then $\rho_m'$ fixes $x$, and so fixes $B$. Thus $B$ also contains $y = v \rho_m'$, which lies in $H_2 \setminus v$. Now, for every $k \geq m$, $\rho_k'$ fixes $x$, and for every $k \leq m-1$, $\rho_k'$ fixes $y$. It follows that every $\rho_k'$ fixes $B$, and so $B = O$ and $\G'$ is primitive.

Now note that $(\rho_{m-1}' \rho_m')^2$ acts as a 3-cycle on $v$ and its neighbors while fixing every other point in $O$. If $n \geq 6$, then \cref{prime-cycle} implies that $\G'$ is either $A_n$ or $S_n$. 
If $n = 3$, then the group is generated by two transpositions and is clearly $S_3$.
If $n = 4$, then either $H_1 \cap O$ or $H_2 \cap O$ consists of 2 vertices and a single edge.  By \cref{sym}, $\G'$ is a symmetric group.

Finally, if $n=5$, then there cannot be two edges of label $m-1$ and two edges of label $m$; that would require at least 7 points. So either $\rho_{m-1}$ or $\rho_m$ is a prime cycle that fixes at least 3 points, so again \cref{prime-cycle} finishes the proof.
\end{proof}

Of course, if we want to make a string C-group by gluing, it is necessary that the two sggis that were glued together are themselves string C. 

\begin{hypothesis}
\label{glued-cpr}
$\G$ satisfies \cref{glued-pr} and furthermore, $\langle \rho_0, \ldots, \rho_{m-1} \rangle$ and $\langle \rho_m, \ldots, \rho_{r-1} \rangle$ are string C-groups.
\end{hypothesis}

We now consider which intersection conditions need to be verified if a group satisfies \cref{glued-cpr}.

\begin{lemma} \label{some-ints_free}
If $\G$ satisfies \cref{glued-cpr} and, for all $0 \leq i \leq m-1$, $\langle \rho_0, \ldots, \rho_{m-1} \rangle 
\cap \langle \rho_i, \ldots, \rho_{r-1} \rangle = \langle \rho_i, \ldots, \rho_{m-1} \rangle$ then for all $0 \leq i \leq j \leq m-1$ $\langle \rho_0, \ldots, \rho_{j} \rangle 
\cap \langle \rho_i, \ldots, \rho_{r-1} \rangle = \langle \rho_i, \ldots, \rho_{j} \rangle$.
\end{lemma}

\begin{proof}
Under the given assumptions,
\begin{align*}
\langle \rho_0, \ldots, \rho_j \rangle \cap \langle \rho_i, \ldots, \rho_{r-1} \rangle &= \langle \rho_0, \ldots, \rho_j \rangle \cap \langle \rho_0, \ldots, \rho_{m-1} \rangle \cap \langle \rho_i, \ldots, \rho_{r-1} \rangle \\
&= \langle \rho_0, \ldots, \rho_j \rangle \cap \langle \rho_i, \ldots, \rho_{m-1} \rangle \\
&= \langle \rho_i, \ldots, \rho_j \rangle, 
\end{align*}
where the last equality follows since $\langle \rho_0, \ldots, \rho_{m-1} \rangle$ is a string C-group.

\end{proof}

Considering the dual of $\G$ then gives us the following.
\begin{corollary} \label{some-ints_free-2}
If $\G$ satisfies \cref{glued-cpr} and, for all $m \leq j \leq r-1$, $\langle \rho_m, \ldots, \rho_{r-1} \rangle 
\cap \langle \rho_0, \ldots, \rho_{j} \rangle = \langle \rho_m, \ldots, \rho_{j} \rangle$ then for all $m \leq j \leq r-1$, $\langle \rho_i, \ldots, \rho_{r-1} \rangle 
\cap \langle \rho_0, \ldots, \rho_{j} \rangle = \langle \rho_i, \ldots, \rho_{j} \rangle$.
\end{corollary}

To understand another collection of intersection conditions, we will examine the action of $\langle \rho_i, \ldots, \rho_j \rangle$ on the orbit of $v$.

\begin{theorem}\label{thm:main}
Suppose that $\G$ is an even permutation group that satisfies \cref{glued-cpr}. Let $1 \leq i \leq m-1$ and  $m \leq j \leq r-2$.  
Then $\langle \rho_0, \ldots, \rho_j \rangle \cap \langle \rho_i, \ldots, \rho_{r-1} \rangle = \langle \rho_i, \ldots, \rho_j \rangle$.  
\end{theorem}

\begin{proof}
We will show that for any element $\varphi$ of $\langle \rho_0, \ldots, \rho_j \rangle \cap \langle \rho_i, \ldots, \rho_{r-1} \rangle$, we can find an element of  $\langle \rho_i, \ldots, \rho_j \rangle$ that acts the same as $\varphi$ does on each point in $G$.

Let $O = v \langle \rho_i, \ldots, \rho_j \rangle$, and let $H_s' = H_s \setminus O$ for $s \in \{1,2\}$. Let $\varphi \in \langle \rho_0, \ldots, \rho_j \rangle \cap \langle \rho_i, \ldots, \rho_{r-1} \rangle$. Let us write $\varphi = \rho_{i_1} \cdots \rho_{i_t}$ with each $i_s \in \{0, \ldots, j\}$. Then, let $\varphi_1$ be this expression for $\varphi$ but with every $\rho_k$ with $k \leq m-1$ omitted. Since each such $\rho_k$ fixes $H_2'$ pointwise, it follows that $\varphi_1$ acts on points in $H_2'$ the same as $\varphi$ does. Furthermore, $\varphi_1$ fixes $H_1'$ pointwise. Note that $\varphi_1$ (and thus $\varphi$) must preserve $H_2'$ setwise, since $\varphi_1 \in \langle \rho_m, \ldots, \rho_j \rangle \leq \langle \rho_i, \ldots, \rho_j \rangle$, and so if $x \in H_2'$ then $x \varphi_1 \not \in O$.

Similarly, write $\varphi = \rho_{j_1} \cdots \rho_{j_l}$ with each $j_s \in \{i, \ldots, r-1\}$, and then set $\varphi_2$ to be the same expression but omitting every $\rho_k$ with $k \geq m$. Then $\varphi_2$ fixes $H_2'$ pointwise and acts on $H_1'$ as $\varphi$ does. By a similar argument to last time, $\varphi$ preserves $H_1'$. Noting that $\varphi_1 \in \langle \rho_m, \ldots, \rho_j \rangle \leq \langle \rho_i, \ldots, \rho_j \rangle$ and $\varphi_2 \in \langle \rho_i, \ldots, \rho_{m-1} \rangle \leq \langle \rho_i, \ldots, \rho_j \rangle$, it follows that $\varphi_1 \varphi_2$ is an element of $\langle \rho_i, \ldots, \rho_j \rangle$ that acts on the points in $H_1' \cup H_2'$ the same way that $\varphi$ does. In particular, since $\varphi$ preserves $H_1' \cup H_2'$, it must also preserve its complement $O$.

Next, we want to show that there is an element $\gamma$ of $\langle \rho_i, \ldots, \rho_j \rangle$ that acts on $O$ the same way that $\varphi$ does. 
By \cref{glued-primitive}, the action of $\langle \rho_i, \ldots, \rho_j \rangle$ on $O$ is either a symmetric group or an alternating group. In the former case, there is nothing to show. In the latter case, we only need to show that the action of $\varphi$ on $O$ is even.  Note that, since $\varphi_1 \varphi_2 \in \langle \rho_i, \ldots, \rho_j \rangle$, which acts as an alternating group on $O$, the restriction of $\varphi_1 \varphi_2$ to $O$ is even. Furthermore, $\G$ itself is even, so the restriction of $\varphi_1 \varphi_2$ to $G \setminus O = H_1' \cup H_2'$ is also even. Then, recalling that $\varphi_1 \varphi_2$ acts on $H_1' \cup H_2'$ the same way that $\varphi$ does, it follows that $\varphi$ itself acts on $H_1' \cup H_2'$ and thus $O$ as an even permutation.

We have found elements $\varphi_1 \varphi_2$ and $\gamma$ in $\langle \rho_i, \ldots, \rho_j \rangle$ such that the former acts like $\varphi$ on $H_1' \cup H_2'$ and the latter acts like $\varphi$ on $O$. It remains to show that there is an element of $\langle \rho_i, \ldots, \rho_j \rangle$ that acts like $\varphi$ on $H_1' \cup H_2'$ and $O$ simultaneously. Considering the action of $\langle \rho_i, \ldots, \rho_j \rangle$ on $G$, we can represent it as the mix of $\Lambda$ and $\Delta$, where $\Lambda$ is the induced permutation group on $H_1' \cup H_2'$ and $\Delta$ is the induced permutation group on $O$. Thus, our goal is to show that $(\varphi_1 \varphi_2, \gamma) \in \Lambda \mix \Delta$. Note that since $\G$ is even, the action of $\varphi$ on $H_1' \cup H_2'$ has the same parity as the action on $O$, and so $\varphi_1 \varphi_2$ has the same parity as $\gamma$. The result then follows from \cref{mix-with-sn} taking $\alpha = (\rho_{m-1} \rho_m)^2$.

\end{proof}

We make two notes following the proof of the theorem. 
First, we note that the bounds for $i$ and $j$ are sharp.  There exist even permutation groups $\G$ that satisfy \cref{glued-cpr} where $\langle \rho_i, \ldots, \rho_j \rangle$ acts as a symmetric group on the orbit of $v$ for all $ 1 \leq i \leq j \leq r-2$ with the property that some $\langle \rho_0, \ldots, \rho_j \rangle \cap \langle \rho_i, \ldots, \rho_{r-1} \rangle \neq \langle \rho_i, \ldots, \rho_j \rangle $ and are thus not string C-groups.  In particular, when $i=m$ or $j=m-1$ the intersections can begin to fail.  For example to build $\G$, take the group with the following permutation representation graph and glue it to its dual at the point $v$

\[
\xymatrix@-0pc{
*+[o][F]{} \ar@{-}[r]^0 \ar@<-.5ex>@{-}[r]_2 & 
*+[o][F]{} \ar@{-}[r]^1 & 
*+[o][F]{} \ar@{-}[r]^0 & 
*+[o][F]{} \ar@{-}[r]^1 & 
*+[o][F]{} \ar@{-}[r]^0 & 
*+[o][F]{} \ar@{-}[r]^1 & 
*+[o][F]{} \ar@{-}[r]^0 & 
*+[o][F]{} \ar@{-}[r]^1 & 
*+[o][F]{} \ar@{-}[r]^2 & 
*+[o][F]{v} 
}
\]

When $i=1$ and $j=2=m-1$, the action of $\langle \rho_1, \rho_2 \rangle$ on the orbit of $v$ is $S_3$.  
However, it can be checked that $\langle \rho_0, \ldots, \rho_2 \rangle \cap \langle \rho_1, \ldots, \rho_5 \rangle$ is three times larger than $\langle \rho_1, \rho_2 \rangle$.  Thus the theorem cannot be extended to include this case.

Second, note that the requirement that $\Gamma$ is even is important.  For instance, if we start with the automorphism group of a 3-cube acting on its 6 faces, we can then glue this to its dual at a vertex $v$ to get the following permutation representation graph for a group $\Gamma$.

$$\includegraphics[width=8cm]{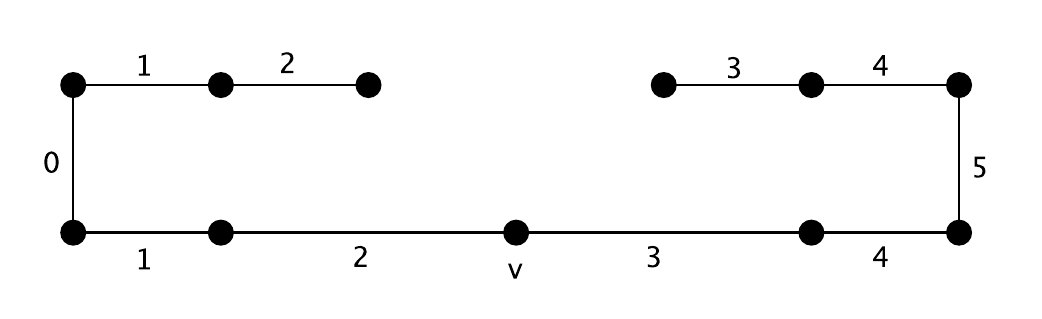}$$

The action of $\langle \rho_i, \ldots, \rho_j \rangle$ on the orbit of $v$ for all $1 \leq i \leq j \leq 4$ is always a symmetric group; however $\Gamma$ is not a string C-group. 
While $\langle \rho_1, \ldots, \rho_5 \rangle$ and $\langle \rho_0, \ldots, \rho_4 \rangle$ are both string C-groups, their intersection is not $\langle \rho_1, \ldots, \rho_4 \rangle$.
In particular one can check that $\langle \rho_1, \ldots, \rho_5 \rangle$ and $\langle \rho_0, \ldots, \rho_4 \rangle$ are both isomorphic to $S_3 \times S_8$ and their intersection is isomorphic to $S_3 \times S_5 \times S_3$.   However, $\langle \rho_1, \ldots, \rho_4 \rangle$ is isomorphic to $(S_3 \times S_5 \times S_3)^+$.

Let us explore another way in which permutation groups satisfying \cref{glued-pr} cannot be string C due to parity considerations. We will see examples of such groups in \cref{sec:moreboxes}.

\begin{proposition}\label{prop:notC}
Suppose $\G$ is a transitive rank $r$ sggi that satisfies \cref{glued-pr}, and that $\langle \rho_0, \ldots, \rho_{m-1} \rangle$ and $\langle \rho_m, \ldots, \rho_{r-1} \rangle$ are not even.
If $\langle \rho_1, \ldots, \rho_{m-1} \rangle$ and $\langle \rho_m, \ldots, \rho_{r-2} \rangle$ each acts evenly on its orbit of $v$, then $\G$ is not a string C-group.
\end{proposition}

\begin{proof}

Since $\G$ is transitive and satisfies \cref{glued-pr}, the orbit of $v$ under $\G_0$ contains all the points in the support of $\langle \rho_m, \ldots, \rho_{r-1} \rangle$.  Since $\langle \rho_m, \ldots, \rho_{r-1} \rangle$ is not even, it must have an odd generator.  Thus $\G_0 = \langle \rho_1, \ldots \rho_m, \ldots, \rho_{r-1} \rangle$ also is odd when restricted to the orbit of $v$ under $\G_0$.  Similarly, by duality, $\G_{r-1}$ is odd when restricted to the orbit of $v$ under $\G_{r-1}$.   By \cref{glued-primitive}, the restriction of these groups to their orbits of $v$ are both symmetric groups.  

Let $O_0$ be the orbit of $v$ under $\G_0$, and let $K_0$ be the subgroup of $\G_0$ consisting of all permutations that fix every point outside $O_0$. The group $K_0$ is the kernel of the homomorphism that restricts the action of $\G_0$ to the complement of $O_0$, which is invariant under $\G_0$, and hence $K_0$ is normal in $\G_0$. Therefore, the action induced by $K_0$ on $O_0$ is normal in the action induced by $\G_0$ on $O_0$. Since the latter action is $S_{O_0}$, the action induced by $K_0$ on $O_0$ is a normal subgroup of $S_{O_0}$.

The group $\langle\rho_m,\ldots,\rho_{r-1}\rangle$ is contained in $K_0$, since it fixes every point outside $O_0$, and it contains an odd permutation. Thus the action induced by $K_0$ on $O_0$ is a normal subgroup of $S_{O_0}$ containing an odd permutation. It follows that $K_0$ acts as a symmetric group on $O_0$. Since $K_0$ fixes every point outside $O_0$, the group $\G_0$ contains the full symmetric group supported on $O_0$.

Dually, $\G_{r-1}$ contains the full symmetric group supported on its orbit $O_{r-1}$ of $v$. Let $\Delta=O_0\cap O_{r-1}$. Every permutation supported on $\Delta$ is contained in both $\G_0$ and $\G_{r-1}$. Therefore,
$S_\Delta\leq \G_0\cap\G_{r-1}.$ Now let $O$ be the orbit of $v$ under $\G_{0,r-1}$. Since $\G_{0,r-1}\leq\G_0\cap\G_{r-1}$, we have $O\subseteq\Delta$.  The group $
\langle\rho_1,\ldots,\rho_{m-1},
\rho_m,\ldots,\rho_{r-2}\rangle
$  acts evenly on $O$. Furthermore, $\langle\rho_1,\ldots,\rho_{m-1}\rangle$ acts evenly on its orbit of $v$ and fixes the remaining points of $O$, while $\langle\rho_m,\ldots,\rho_{r-2}\rangle$ acts evenly on its orbit of $v$ and fixes the remaining points of $O$.

Choose two distinct points $x,y\in O$. Since $O\subseteq\Delta$, the transposition $(x,y)$ belongs to $\G_0\cap\G_{r-1}$. However, this transposition acts oddly on $O$, so it does not belong to $\G_{0,r-1}$. Therefore,
$\G_0\cap\G_{r-1}\neq\G_{0,r-1},$
and hence $\G$ is not a string C-group.
\end{proof}

Finally, we come to our main theorem for this section. 
\begin{theorem}
\label{string-C}
Let $\G$ be an even permutation group that satisfies \cref{glued-cpr}. Suppose that $\G$ satisfies the following:
\begin{enumerate}
\item For all $1 \leq i \leq m-2$, $\langle \rho_0, \ldots, \rho_{m-1} \rangle \cap \langle \rho_i, \ldots, \rho_{r-1} \rangle = \langle \rho_i, \ldots, \rho_{m-1} \rangle$,
\item For all $m+1 \leq j \leq r-2$, $\langle \rho_m, \ldots, \rho_{r-1} \rangle \cap \langle \rho_0, \ldots, \rho_j \rangle = \langle \rho_m, \ldots, \rho_j \rangle$.
\end{enumerate}
Then $\G$ is a string C-group.
\end{theorem}

\begin{proof}
This almost directly follows from \cref{int-cond}, \cref{some-ints_free}, \cref{some-ints_free-2}, and \cref{thm:main}, except that to apply \cref{some-ints_free} and \cref{some-ints_free-2}, we still need to handle the cases $i = m-1$ and $j = m$. That is, it remains to show that
\[ \langle \rho_0, \ldots, \rho_{m-1} \rangle \cap \langle \rho_{m-1}, \ldots, \rho_{r-1} \rangle = \langle \rho_{m-1} \rangle \]
and 
\[ 
\langle \rho_m, \ldots, \rho_{r-1} \rangle \cap \langle \rho_0, \ldots, \rho_m \rangle = \langle \rho_m \rangle.
\]
We will show the latter; the former then follows by duality. 

Let $y$ be the neighbor of $v$ in $H_2$, which is connected to $v$ by an edge of label $m$. 
Let $\varphi \in \langle \rho_m, \ldots, \rho_{r-1} \rangle \cap \langle \rho_0, \ldots, \rho_m \rangle$. Since $\varphi \in \langle \rho_m, \ldots, \rho_{r-1} \rangle$, it fixes $H_1 \setminus v$ pointwise. Now, write $\varphi = \rho_{k_1} \cdots \rho_{k_t}$, with each $k_s \in \{0, \ldots, m\}$. If $\rho_m$ occurs an even number of times, then $\varphi$ fixes $H_2 \setminus \{v,y\}$ pointwise as well, so all it could possibly do is swap $v$ and $y$. Since $\G$ is even, $\varphi$ must be the identity. 

On the other hand, suppose $\rho_m$ occurs an odd number of times. Then for every $x \in H_2 \setminus \{v,y\}$ we have $x \varphi = x \rho_m$. Now, since $\G$ is even, there are an even number of edges of label $m$. Thus $\varphi$ cannot fix $v$, since then it would be the product of an odd number of transpositions. It follows that $v \varphi = y$, and so $\varphi$ acts like $\rho_m$ everywhere. So no matter the parity of the number of occurrences of $\rho_m$, we have $\varphi \in \langle \rho_m \rangle$. 
\end{proof}

\section{Families of sggis for alternating groups} \label{sec:families}

In order to build polytopes of high rank for the alternating group, we will need some initial ingredients that we can glue together. The first such group is described in \cite[Lemma 6.6]{FLM12}, which provides a family of permutation representation graphs for rank $m$ string C-groups acting as permutation groups of degree $n=2m+2$ which are isomorphic to $(2 \wr S_{m+1})^+$ for $m \geq 4$.   The permutation representation graph for the dual of such a group is shown in \cref{fig:xyz}.  We note that this family of permutation representation graphs also gives a string C-group when $m=3$.  However, the group is no longer isomorphic to a wreath product.  When $m=3$ the group is isomorphic to $C_2 \times S_4$ and the permutation representation is that of the automorphism group of a regular map $(3,6)_{(2,0)}$ acting on its 8 faces.  We will denote this family of string C-groups as $X(m)$.

\begin{figure}[h]
\begin{center}
$X(m)$

\includegraphics[width=8cm]{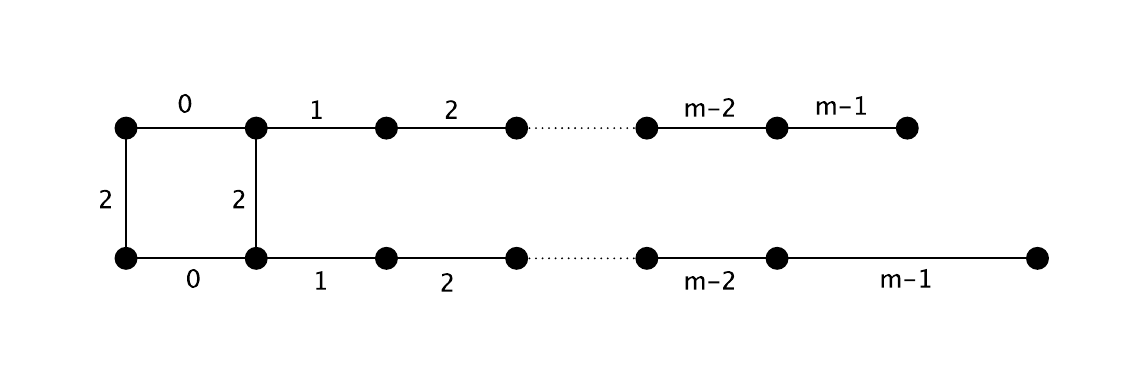}

$Y(m)$

\includegraphics[width=8cm]{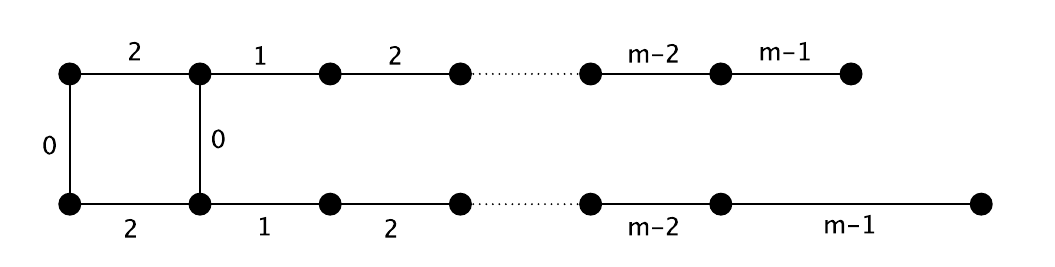}

$Z(m)$

\includegraphics[width=8cm]{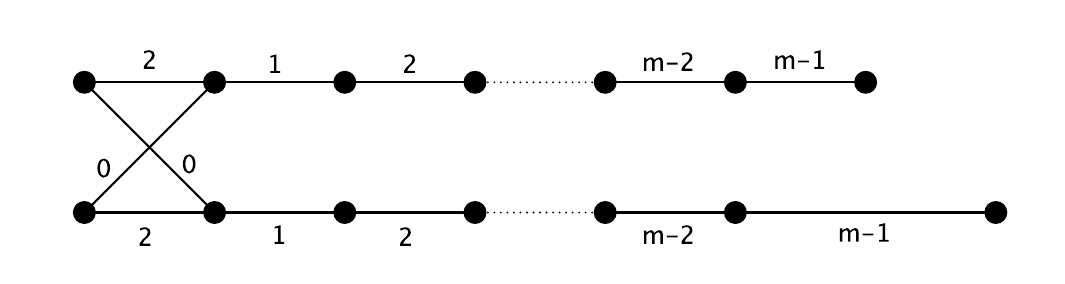}
\end{center}
\caption{The sggis $X(m)$, $Y(m)$, and $Z(m)$}
\label{fig:xyz}
\end{figure}

We will also consider here some closely related groups, which will also be sggis of rank $m$ and degree $n = 2m+2$. These will be denoted $Y(m)$ and $Z(m)$. See \cref{fig:xyz}.

Let us show that $Y(m)$ and $Z(m)$ are also string C-groups typically acting as even subgroups of wreath products.

\begin{proposition} \label{Y/2}
For all $m \geq 3$, if $\G = \langle \rho_0, \ldots \rho_{m-1} \rangle$ is the rank $m$ sggi $Y(m)$ or $Z(m)$, then $\G$ is a string C-group.  Furthermore, when $m \ge 4$, $\G$ is isomorphic to $(2 \wr S_{m+1})^+$.
\end{proposition}

\begin{proof}
We proceed by induction. The group $Y(3)$ is the automorphism group of the regular map $(4,4)_{(2,2)}$ acting on its facets, and $Z(3)$ is the automorphism group of the regular octahedron acting on its facets. These are both string C-groups. The cases $m = 4$ and $m=5$ can be checked using GAP. Now assume $m >5$.   The group $\G_{m-1} =  \langle \rho_0, \ldots \rho_{m-2} \rangle$ is a string C-group isomorphic to $(2 \wr S_{m})^+$ by the inductive hypothesis.  The groups $\G_{0} =  \langle \rho_1, \ldots \rho_{m-1} \rangle$ and $\G_{0,m-1} =  \langle \rho_1, \ldots \rho_{m-2} \rangle$ were shown to be string C-groups isomorphic to $S_{m+1}$ and $S_{m}$ in \cite[Lemma 20]{FL11}.   Since $\G_{0,m-1}$ is a maximal subgroup of $\G_{0}$, and $\rho_0 \not \in \G_0$, it follows from Lemma~\ref{max} that $\G$ is a string C-group.

To see that $\G \cong (2 \wr S_{m+1})^+$, we notice that $\G$ has one more block of size 2 when compared to $\G_{m-1}$.  The action on the $(m+1)$ blocks of size 2 is the symmetric group as $\G_0$ acts between the blocks.  It remains to show that any possible even permutation in the groups acting within the blocks is possible, so that the base group is $(C_2^{m+1})^+$.  By induction, we know that any possible even permutation within the first $m$ blocks is possible using only elements in $\G_{m-1}$.  Let $\alpha$ be the permutation in $\G_{m-1}$ that swaps the points within blocks $m-1$ and $m$, the last two blocks moved by $\G_{m-1}$.    Then $\beta = \rho_{m-1} \alpha \rho_{m-1}$ swaps the points in blocks $m-1$ and $m+1$. Multiplying $\beta$ by all even permutations in $\G_{m-1}$ then gives us all possible permutations in $(C_2^{m+1})^+$, and so $\G \cong (2 \wr S_{m+1})^+$.
\end{proof}

Let $\G \in \{ X(m), Y(m), Z(m) \}$, and let $\G' \in \{ X(r_2), Y(r_2), Z(r_2) \}$ for some values of $m,r_2 \geq 3$.  We will consider the nine families of sggis obtained when you glue $\G$ to the dual of $\G'$.  For instance, we could glue $X(m)$ to the dual of $Y(r_2)$; we denote the result as $XY(m,r_2)$.  We will use a similar naming convention for all possible ways to glue $\G$ and the dual of $\G'$.  We note that some of these families are dual to each other; for instance, $XZ(m,r_2)$ is dual to $ZX(r_2,m)$.  The results are all sggis of rank $r=m+r_2$ for a permutation group of degree $n=2m+2r_2+3$. Note that the groups $XX(m,m)$, $YY(m,m)$ and $ZZ(m,m)$ are all self-dual sggis; see \cref{fig:glue-xyz}. Our goal is to show that these groups are all string C-groups.

\begin{figure}[h]
\begin{center}
$XX(m,m)$

\includegraphics[width=16cm]{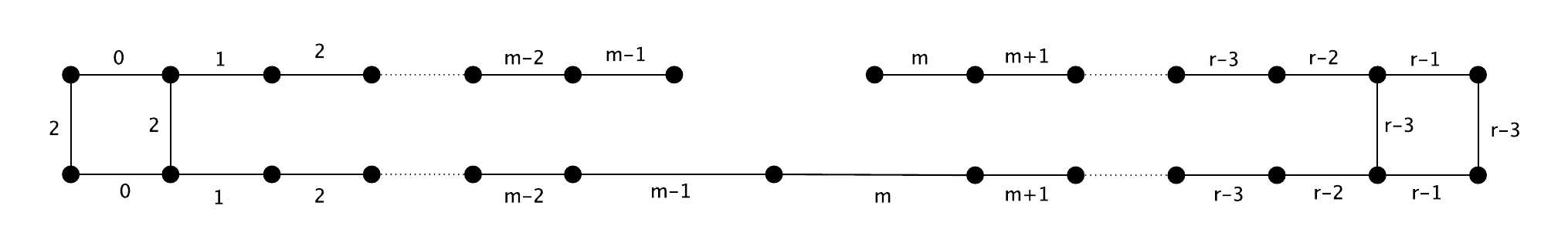}

$YY(m,m)$

\includegraphics[width=16cm]{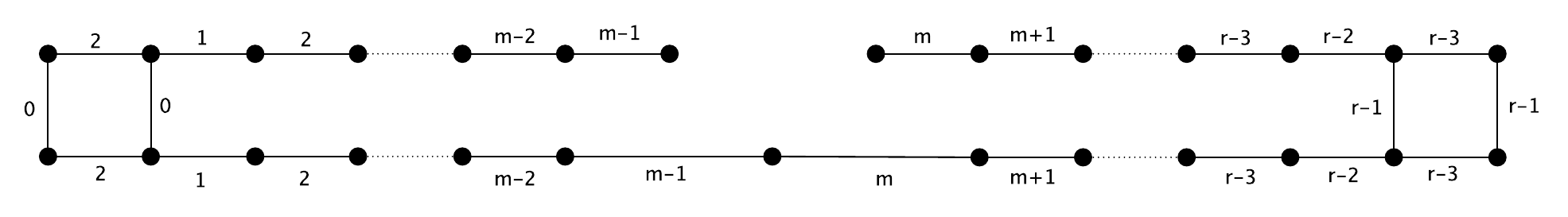}

$ZZ(m,m)$

\includegraphics[width=16cm]{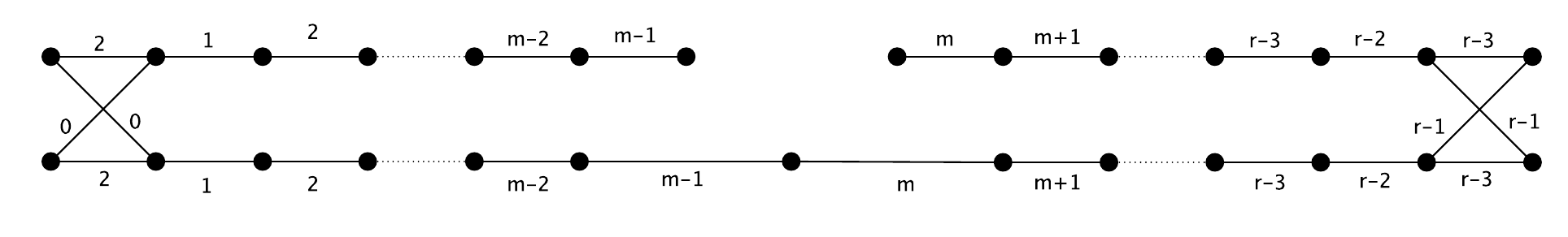}

\end{center}
\caption{The self-dual sggis $XX(m,m), YY(m,m)$, and $ZZ(m,m)$, where $r = 2m$.}
\label{fig:glue-xyz}
\end{figure}

\begin{lemma} 
\label{glueXYZ}
Let \(\G=\langle \rho_0,\ldots,\rho_{r-1}\rangle\) be an even sggi satisfying
\cref{glued-cpr}, where $\langle \rho_m, \ldots, \rho_{r-1} \rangle$ is the dual of one of
\(X(r_2)\), \(Y(r_2)\), or \(Z(r_2)\), with \(r_2\geq 3\). Then, for every
\(m+1\leq j\leq r-2\),
\[
\langle \rho_0,\ldots,\rho_j\rangle
\cap
\langle \rho_m,\ldots,\rho_{r-1}\rangle
=
\langle \rho_m,\ldots,\rho_j\rangle.
\]
\end{lemma}

\begin{proof}

Let $\Lambda = \langle \rho_0, \ldots, \rho_j \rangle \cap \langle \rho_m, \ldots, \rho_{r-1} \rangle$.   It is clear that $ \langle \rho_m, \ldots, \rho_j \rangle \leq \Lambda $, so we just need to show the reverse inclusion.

When $r_2 \geq 4$,  \cite[Lemma 6.6]{FLM12} and \cref{Y/2} show that $X(r_2)$, $Y(r_2)$, and $Z(r_2)$ are isomorphic to $(2 \wr S_{r_2+1})^+$.  Let's consider these cases first, and conclude by examining the cases where $r_2 = 3$.

When we glue any even string C-group to the dual of $X(r_2)$, we say the result is of type $*X$. (Note that in order for $\G$ to be even, the group we glue to $X(r_2)$ must also be even.) Similarly groups of type $*Y$ and $*Z$ are the result of gluing to the dual of $Y(r_2)$ and $Z(r_2)$, respectively.  
First, suppose that $\G$ is type $*Y$ or $*Z$.   
Here $\langle \rho_m, \ldots, \rho_j \rangle$ fixes $H_1 \setminus v$ pointwise, acts trivially within the blocks of $H_2$, and acts as the full symmetric group $S_k$ on the $k$ blocks that it does not fix (using, for example, \cref{sym}). Now,  $\langle \rho_m, \ldots, \rho_{r-1} \rangle$ also fixes $H_1 \setminus v$ pointwise, while acting as an imprimitive group with blocks of size 2 on $H_2$. Since $\Lambda \leq \langle \rho_m, \ldots, \rho_{r-1} \rangle$, then $\Lambda$ does the same. 
On the other hand, $\Lambda \leq \langle \rho_0, \ldots, \rho_j \rangle$, which acts trivially within the blocks of $H_2$, and so $\Lambda$ is a subgroup of the symmetric group on the blocks that are not fixed; these are precisely the $k$ blocks that $\langle \rho_m, \ldots, \rho_j \rangle$ moves. It follows that $\Lambda \leq \langle \rho_m, \ldots, \rho_j \rangle$.

Now suppose $\G$ is type $*X$.

If $j \leq r-4$, then the previous argument about type $*Y$ or $*Z$ also works for $*X$. Suppose $j = r-3$. Let $T$ be the set of points of $H_2$ not in the support of $\rho_{r-1}$. The group $\langle \rho_m, \ldots, \rho_{r-3} \rangle$ has a permutation representation graph that is two copies of a permutation representation graph of a simplex and two $(r-3)$-edges. Thus, the group is a subgroup of $S_k \times S_2$, where $k$ is the number of vertices in one of the large connected components.   Since $(\rho_{r-4} \rho_{r-3})^3$ fixes points in the two copies of the permutation representation graph of the simplex, and moves the points on the $(r-3)$-edges we know that the mix is larger than $S_k$ and thus
 $\langle \rho_m, \ldots, \rho_{r-3} \rangle \cong S_k \times S_2$. Similarly, $\Lambda$ fixes $H_1 \setminus v$ pointwise and acts trivially on the blocks within $T$, while acting as $S_2$ on the support of $\rho_{r-1}$. It follows that $\Lambda \leq \langle \rho_m, \ldots, \rho_{r-3} \rangle$. 

Next, let $j = r-2$.  It follows from \cref{wreath}
that the group $\langle \rho_m, \ldots, \rho_{r-2} \rangle$ is $E_{r-3} \mix (2 \wr S_k)$.
Since the resulting group is even, the transposition for the single edge cannot be in the resulting group, and thus  $\langle \rho_m, \ldots, \rho_{r-2} \rangle \cong 2 \wr S_k \cong  ((2 \wr S_k) \times 2)^+$. Now consider the restriction of $\Lambda$ to each of its two nontrivial orbits.  The restriction of $\Lambda$ to the orbit of $v$ respects the blocks of $\langle \rho_m, \ldots, \rho_{r-1} \rangle$ and thus $\Lambda$ restricted to the orbit of $v$ is a subgroup of $2 \wr S_k$.  The other non-trivial orbit of $\Lambda$ has size two.  Since $\Lambda \leq \G$ is even, $\Lambda \leq ((2 \wr S_k) \times 2)^+$, and we are done.

Finally, suppose that $r_2 = 3$, so that $r = m+3$. We need only consider $j = m+1 = r-2$.

The groups $X(3)$, $Y(3)$, and $Z(3)$ are all even imprimitive groups with 4 blocks of size 2.  The action of the group $Y(3)$ on the blocks is given by the dihedral group $D_4$ (of order 8), whereas for $X(3)$ and $Z(3)$, the action on the blocks is $S_4$.  

It can be verified in GAP that
$X(3)$ and $Z(3)$ are isomorphic to the direct product $S_4 \times S_2$, and
 $Y(3)$ is the even subgroup of the full wreath product $(2 \wr D_4)$.
 
 Let $\G = *X(3)$.  
 Since $\Lambda \leq \langle \rho_m, \rho_{m+1}, \rho_{m+2} \rangle \cong S_4 \times 2$, it follows that $\Lambda$ fixes $H_1\setminus v$ and acts imprimitively on $H_2$ with the blocks of $X(3)$.   Additionally, since 
$\Lambda \leq \langle \rho_0, \ldots, \rho_{m+1} \rangle$, it follows that $\Lambda$ fixes one block of $H_2$.  Thus $\Lambda \leq S_3 \times 2$ acting on the remaining three blocks.  Since $S_3 \times 2 \cong D_6 \cong \langle \rho_m, \rho_{m+1} \rangle$, we are done with this case.
 
 Let $\G = *Y(3)$.
 Since $\Lambda \leq \langle \rho_m, \rho_{m+1}, \rho_{m+2} \rangle$, it follows that $\Lambda$ fixes $H_1\setminus v$ and acts imprimitively on $H_2$ with the blocks of $Y(3)$, where the action of $\Lambda$ on those blocks is a subgroup of $D_4$, the action of $Y(3)$ on the blocks.
 Since $\Lambda \leq \langle \rho_0, \ldots, \rho_{r-2} \rangle$, the action of $\Lambda$ in the blocks is trivial.  
Thus $\Lambda \leq D_4 \cong \langle \rho_m, \rho_{m+1} \rangle.$
 
Let $\G = *Z(3)$.  Since $\Lambda \leq \langle \rho_m, \rho_{m+1}, \rho_{m+2} \rangle$, it follows that $\Lambda$ fixes $H_1\setminus v$ and acts imprimitively on $H_2$ with the blocks of $Z(3)$, where the action of $\Lambda$ on those blocks is a subgroup of $S_4$.   In the permutation representation graph of $Z(3)$ there are two natural rows.  It can be verified in GAP that the subgroup of $ \langle \rho_m, \rho_{m+1}, \rho_{m+2} \rangle$ preserving these two rows setwise is precisely the subgroup $\langle \rho_m, \rho_{m+1} \rangle$.  Since $\Lambda \leq \langle \rho_0, \ldots, \rho_{m+1} \rangle$, and since the only generator in $\langle \rho_m, \rho_{m+1}, \rho_{m+2} \rangle$ that does not preserve the two rows is $\rho_{m+2}$, every element of $\Lambda$ preserves the two rows of $Z(3)$ setwise.  Thus $\Lambda \leq \langle \rho_m, \rho_{m+1} \rangle$.

\end{proof}

\begin{theorem}\label{thm:9families}
Let $\G$ be any sggi of the form: $XX(m,r_2)$,  $XY(m,r_2)$,  $XZ(m,r_2)$,  $YX(m,r_2)$,  $YY(m,r_2)$,  $YZ(m,r_2)$,  $ZX(m,r_2)$,  $ZY(m,r_2)$, or  $ZZ(m,r_2)$.
Then $\G$ is a string C-group isomorphic to $A_n$ for $n=2m+2r_2+3$.  
\end{theorem}

\begin{proof}
    Let $\G = \langle \rho_0, \ldots , \rho_{r-1} \rangle$ have the form of one of the nine stated families.    By construction, $\G$ satisfies \cref{glued-cpr}, since both glued pieces are string C-groups by  \cite[Lemma 6.6]{FLM12} and \cref{Y/2}.   
From \cref{glued-primitive}, we know that $\G$ is an alternating or symmetric group on the orbit of $v$, and since $\G$ is even and transitive, we see that $\G \cong A_n$.  To show that $\G$ is a string C-group, we use \cref{string-C}.  The second criterion in \cref{string-C} follows from \cref{glueXYZ}.  Since the dual of each of these families is again one of the nine families, the first criterion follows from applying \cref{glueXYZ} to the dual of $\G$.  
\end{proof}

\begin{corollary} \label{cor:3examples}
There are at least three self-dual string C-groups of rank $\frac{n-3}{2}$ isomorphic to $A_n$ when $n \geq 15$ and $n \mod 4 = 3$.  \end{corollary}

\begin{proof}
The families $XX(m,m), YY(m,m)$, and $ZZ(m,m)$ are all self-dual string C-groups of rank $r = 2m$ acting as $A_n$ with $n = 4m+3$.
They are pairwise non-isomorphic as can be seen from their types:  $XX(m,m)$ has type $\{3,6^3,3\}$ when $m=3$, and $\{3,6^2, 3^{m-4}, 6 , 3^{m-4}, 6^2, 3\}$ otherwise;
$YY(m,m)$ has type $\{4^2,6,4^2\}$ when $m=3$ and $\{4^2, 6, 3^{m-4},6,3^{m-4},6,4^2\}$ otherwise;
$ZZ(m,m)$ has type $\{3,4,6,4,3\}$ when $m=3$ and $\{3,4,6,3^{m-4},6,3^{m-4},6,4,3\}$ otherwise.

\end{proof}

\section{A more general family of sggis} \label{sec:moreboxes}
The groups $X(m)$ gave us a one-parameter family of string C-groups that could be glued to their duals to construct self-dual string C-groups for the alternating group.  In fact, we can generalize these groups to be part of a larger two-parameter family that will similarly be useful to construct self-dual string C-groups for the alternating group.

We consider the following permutation representation graph on $n = 2b+2t+2$ nodes, with $b \geq 1$ ``boxes'' and with ``tails'' of length $t \geq 2$:
\[ 
\xymatrix@-1pc{
*+[o][F]{1} \ar@{-}[r]^0 \ar@{-}[d]^{b+1} 
& *+[o][F]{3} \ar@{-}[r]^1 \ar@{-}[d]^{b+1} 
& *+[o][F]{5} \ar@{-}[d]^{b+1} \ar@{--}[rr]  
&& *+[o][F]{2b-1} \ar@{-}[r]^{b-1} \ar@{-}[d]^{b+1}
& *+[o][F]{2b+1} \ar@{-}[r]^b \ar@{-}[d]^{b+1} 
& *+[o][F]{2b+3} \ar@{-}[r]^{b+1} 
& *+[o][F]{2b+5} \ar@{-}[r]^{b+2} 
& *+[o][F]{2b+7} \ar@{--}[rr] 
&&  *+[o][F]{n-3} \ar@{-}[rr]^{b+t-1}
&&  *+[o][F]{n-1} \\
*+[o][F]{2} \ar@{-}[r]^0 
& *+[o][F]{4} \ar@{-}[r]^1
& *+[o][F]{6} \ar@{--}[rr]
&& *+[o][F]{2b} \ar@{-}[r]^{b-1} 
& *+[o][F]{2b+2} \ar@{-}[r]^b 
& *+[o][F]{2b+4} \ar@{-}[r]^{b+1}
& *+[o][F]{2b+6} \ar@{-}[r]^{b+2}
& *+[o][F]{2b+8} \ar@{--}[rr]
&& *+[o][F]{n-2} \ar@{-}[rr]^{b+t-1}
&& *+[o][F]{n}
}
\]

We will call this graph $G(b,t)$ and the corresponding group $\G(b,t)$.   Note that $\G(1,t) = X(t+1)$. Our first goal is to establish that $\G(b,t)$ is a string C-group, and to determine its structure.

\begin{proposition} \label{Gbt-base-case}
$\G(b,2)$ is a string C-group for all $b \geq 1$.
\end{proposition}

\begin{proof}
The result can be verified for $b=1$ using GAP.

Let $\G = \G(b,2)$ with $b \geq 2$ and suppose that $\G(k,2)$ is a string C-group for every $k < b$. First, note that $\G_{r-1}=\G_{b+1} \cong S_{b+2}$ and that this is the usual representation of a simplex and is clearly a string C-group. 

Next, we see that $\G_0$ is a mix of (a shifted version of) $\G(b-1,2)$ with $E_{b+1}$. By inductive hypothesis, $\G(b-1,2)$ is a string C-group, and by Lemma~\ref{mix-edge}, so is the mix. Thus $\G_0$ and $\G_{r-1}$ are both string C-groups. Then since $\G_{0,b+1} \cong S_{b+1}$, which is maximal in $\G_{b+1} \cong S_{b+2}$, and clearly $\rho_0 \not \in \G_0$, \cref{max} says that $\G$ is a string C-group.
\end{proof}

\begin{proposition} \label{b-orientable}
Let $b \geq 1$ and $t \geq 2$. Then $\G(b,t)$ is $(b+1)$-orientable if and only if $b$ is even or $t$ is $2$.
\end{proposition}

\begin{proof}

First, suppose $t = 2$. Then every vertex is incident to an edge of label $b+1$, and if we color the vertices $\{1, 3, \ldots, 2b+3\}$ and $2b+6$ black and the remaining vertices white, then $\rho_{b+1}$ reverses the colors while every other $\rho_j$ fixes the colors, so $\G(b,t)$ is $(b+1)$-orientable by \cref{k-coloring}. 

If $b$ is even, then $\rho_{b+1}$ is the only odd permutation, and so $\G$ is $(b+1)$-orientable by \cref{mix-is-iso}.

Finally, suppose $b$ is odd and $t \geq 3$. Let $\sigma = \rho_0 \rho_1 \cdots \rho_{b+1} \rho_{b+2}$. If $b = 1$, we have $\sigma = (1,9,7,5,4)(2,10,8,6,3)$ which has odd order. When $b \geq 3$, $\sigma$ fixes vertex $k$ for $k \geq 2b+9$. For odd $k$ between $3$ and $2b+3$, we have $\sigma(k) = k-1$, and for even $k$ between $4$ and $2b+4$, we have $\sigma(k) = k-3$. Then, noting that $2b+2$ is a multiple of $4$ since $b$ is odd, and $\sigma^2(k) = k-4$ for all $5 \leq k \leq 2b+4$, we get that
\[ \sigma = (1, 2b+7, 2b+5, 2b+3, 2b+2, 2b-1, 2b-2, \ldots, 8, 5, 4) (2, 2b+8, 2b+6, 2b+4, 2b+1, 2b, 2b-3, \ldots, 6, 3). \]

Thus, $\sigma$ is the product of disjoint $(b+4)$-cycles, so it has order $b+4$, which is odd. Then the relation $\sigma^{b+4} = 1$ can be solved for $\rho_{b+1}$ to give an expression for $\rho_{b+1}$ as a word in $\G^{(b+1)}$, showing that $\G$ is not $(b+1)$-orientable by \cref{k-orientable}.
\end{proof}

\begin{theorem} \label{thm:boxes}
For all $b \geq 1$ and $t \geq 2$, the group $\G(b,t)$ is string C.  Furthermore, when $t=2$, $\G(b,t) \cong 2 \times S_{b+3}$.  When $t\geq 3$ and $b$ is even then $\G(b,t) \cong 2 \wr S_{b+t+1}$, and when $t\geq 3$ and $b$ is odd then $\G(b,t) \cong (2 \wr S_{b+t+1})^+$.
\end{theorem}

\begin{proof}
First we prove that the groups have the given isomorphism types. 
Let $t=2$.  The element $z =(\rho_{b+1} \rho_{b})^3$ has no action between the blocks and interchanges the elements in each of the $b+3$ blocks of size two.  Also $z$ commutes with all the generators of $\G$.  Consider the quotient $\bar{\G} = \G / \langle z \rangle = \langle \bar{\rho_0}, \ldots, \bar{\rho_{b+1}} \rangle$, where $\bar{\rho_i}$ is the image of $\rho_i$ under the quotient map.  The group $\langle \bar{\rho_0}, \ldots, \bar{\rho_{b+1}} \rangle$ is still an sggi.  By the choice of $z$, $(\bar{\rho_{b+1}} \bar{ \rho_{b}})^3 = 1_{\bar{\G}}$. Thus for all the generators of $\bar{\G}$ we have that $(\bar{\rho_i} \bar{ \rho_{i+1}})^3 = 1_{\bar{\G}}$, and $\bar{\G}$ satisfies the exact same relations as the simplex.  Since the simplex has no nontrivial quotients that preserve these relations, $\bar{\G} \cong S_{b+3}$.  
Since $\G$ has a subgroup $ \langle \rho_0, \ldots, \rho_{b}, z \rho_{b+1} \rangle$ isomorphic to $S_{b+3}$, it follows that $\G$ is a direct product.  $\G \cong 2 \times S_{b+3}$.

Now assume $t \geq 3$.  Again the element $u =(\rho_{b+1} \rho_{b})^3$ has no action between the blocks and interchanges the elements in the first $b+3$ blocks of size two.   However, now $u$ acts trivially outside the support of $\rho_b$ and $\rho_{b+1}$.  We can say that $u = \alpha_1 \alpha_2 \cdots \alpha_{b+3}$ where $\alpha_i$ represents a swapping of the two points in block $i$. Since the action on the blocks is the full symmetric group, conjugates of $u$ give products of length $(b+3)$ of different $\alpha_i$.   These generate the full elementary abelian group generated by the $\alpha_i$ when $b+3$ is odd, and generate the even subgroup when $b+3$ is even.
Thus $\G$ is the full wreath product $2 \wr S_{b+t+1}$ when $b+3$ is odd, and is the even subgroup of the wreath product when $b+3$ is even.

Next, we proceed by induction on $b+t$ to show that $\G(b,t)$ is string C.
\cref{Gbt-base-case} establishes all the cases where $t = 2$, while \cite[Lemma 6.6]{FLM12} finishes all the cases where $b=1$.

Let $\G = \G(b,t)$ with $t \geq 3$ and $b \geq 2$, and assume that $\G(c,k)$ is string C for all $c+k < b+t$.
Note that $\G_{b+t-1} \cong \G(b,t-1)$, which is string C by the inductive hypothesis.  

Next we show that $\G_0$ is string C. We see that $\G_0 \cong \G(b-1,t) \mix E_b$, and by inductive hypothesis, $\G(b-1,t)$ is string C. If $b$ is odd, then \cref{b-orientable} tells us that $\G(b-1,t)$ is $b$-orientable, and then $\G_0$ is string C by \cref{mix-ek-string-c}. If $b$ is even, then consider $\G(b-1,t)_0 \cong \G(b-2,t) \mix E_{b-1}$. This is $(b-1)$-orientable by \cref{k-mix}, and so again \cref{mix-ek-string-c} says that $\G_0$ is string C.

Let $\Lambda = \G_0 \cap \G_{r-1}$. By \cite[Proposition 2E16]{ARP}, it remains to show that $\Lambda \leq \G_{0,b+t-1}$. Note that whatever the values of $b$ and $t$, the group $\G$ has a block structure with $b+t+1$ blocks of size 2. Furthermore, $\G_0$ (and thus $\Lambda$) always fixes the first block setwise, and $\G_{r-1}$ (and thus $\Lambda$) fixes the last block pointwise. Thus $\Lambda$ will be contained in the pointwise stabilizer of the last block under the action of $\G_0$.

First suppose $t = 3$. Then
\[ \G_{0,r-1} \cong E_{b} \mix \G(b-1, 2) \cong E_{b} \mix (2 \times S_{b+2}). \]
Since $\G(b-1,2)$ is $b$-orientable by \cref{b-orientable}, it follows from \cref{k-orientable} that $\G_{0,r-1} \cong 2 \times S_{b+2}$. Now, $\G_{r-1} \cong \G(b,t-1) \cong 2 \times S_{b+3}$, so $\G_{0,r-1}$ is a maximal subgroup of $\G_{r-1}$. Furthermore, clearly $\rho_0 \not \in \G_0$, so by \cref{max}, we have that $\G$ is string C in this case.

Now suppose $t \geq 4$. If $b$ is odd, then 
\[ \G_{0,r-1} \cong E_{b+1} \mix (2 \wr S_{b+t-1}). \]
(Here we are not reindexing the way that we did in the last case, so we get $E_{b+1}$ instead of $E_b$.)
Similarly,
\[ \G_0 \cong E_{b+1} \mix (2 \wr S_{b+t}). \]
Clearly the pointwise stabilizer in $\G_0$ of the last block is $E_{b+1} \mix (2 \wr S_{b+t-1})$, and so $\Lambda$ is contained in $\G_{0,r-1}$ in this case.

Finally, suppose $b$ is even. Then
\[ \G_{0,r-1} \cong E_{b+1} \mix (2 \wr S_{b+t-1})^+ \]
and 
\[ \G_0 \cong E_{b+1} \mix (2 \wr S_{b+t})^+. \]
Again, the pointwise stabilizer in $\G_0$ of the last block is precisely $E_{b+1} \mix (2 \wr S_{b+t-1})^+$, so $\Lambda$ is contained in this stabilizer, which is $\G_{0,r-1}$.

\end{proof}

\subsection{Gluing to make more alternating and symmetric groups}

If we glue $\G(b_1,t_1)$ to the dual of $\G(b_2,t_2)$ (at vertex $2b_1+2t_1+2$) we get a new group $\G(b_1,t_1,b_2,t_2)$ for $b_i \geq 1$ and $t_i \geq 2$ which is a permutation group of degree $n= 2b_1+2t_1 + 2b_2 + 2t_2 + 3$.  Our goal is to show that this is a string C-group. In fact, we will show something more general: that gluing $\G(b,t)$ to any even permutation group yields a string C-group.

\begin{lemma}\label{glue-box-left}
Let $\G=\langle \rho_0,\ldots,\rho_{r-1}\rangle$ be an sggi satisfying \cref{glued-pr}, where $\langle \rho_0,\ldots,\rho_{m-1}\rangle$ is one of
the groups $\G(b,t)$, with $b \geq 1$ and $t\geq2$. Suppose that $\langle \rho_m, \ldots, \rho_{r-1} \rangle$ is even and a string C-group. Then, for every $1\leq i\leq m-2$,
\[
\langle \rho_0,\ldots,\rho_{m-1}\rangle
\cap
\langle \rho_i,\ldots,\rho_{r-1}\rangle
=
\langle \rho_i,\ldots,\rho_{m-1}\rangle .
\]
\end{lemma}

\begin{proof}
First, note that $\G(b,t)$ is string C by \cref{thm:boxes}, and so $\G$ satisfies \cref{glued-cpr}. If $b = 1$, then the claim follows from \cref{glueXYZ} applied to the dual of $\G$, and so we can assume that $b \geq 2$. 

Let $L = \langle \rho_0, \ldots, \rho_{m-1} \rangle$, $R_i = \langle \rho_i,\ldots,\rho_{r-1}\rangle$, and $I_i = L \cap R_i$. Let $M_i = \langle \rho_i, \ldots, \rho_{m-1} \rangle$. Our goal is to show that $I_i \leq M_i$ for $1 \leq i \leq m-2$; the reverse inclusion is obvious.

Since $I_i \leq L$, it preserves the same block structure as $L$, and it fixes $H_2 \setminus v$ pointwise. Since $I_i \leq R_i$, it stabilizes each of the first $i$ blocks setwise. Furthermore, if $i > b+1$, then $I_i$ must fix these first $i$ blocks pointwise, and otherwise the only nontrivial action within blocks is to interchange the two points in every block simultaneously. In particular, this shows that $I_i \leq E_{b+1} \mix (2 \wr S_{b+t+1-i})$ for every $1 \leq i \leq m-2$. (To be clear, we determine the action of $I_i$ on the first $i$ blocks by writing each element of $I_i$ as an element of $L$; then the parity of the number of occurrences of $\rho_{b+1}$ in that word determines whether to fix those blocks pointwise or interchange the elements in each one.)

Suppose $t=2$.  \cref{thm:boxes} showed that $L \cong 2 \times S_{b+3}$ and \cref{b-orientable} showed $L$ is $(b+1)$-orientable.
Since $m-2=b$, we need to consider $1 \leq i \leq b$.  We claim that $M_i \cong 2 \times S_{b+3-i}$.  
When $i=b$ this can easily be verified in GAP.  When $i < b$, then $M_i$ is the mix of $E_{b+1}$ with a group isomorphic to $\G(b-i,2)$, which is $(b+1)$-orientable by \cref{b-orientable}. Then since $\G(b-i,2) \cong 2 \times S_{b+3-i}$ by \cref{thm:boxes}, it follows that $M_i \cong 2 \times S_{b+3-i}$ by \cref{k-orientable}. Since $L \cong 2 \times S_{b+3}$, the stabilizer of the first $i$ blocks is contained in $2 \times S_{b+3-i}$, so $I_i \leq M_i$.

Let $t \geq 3$. There are four different cases to consider based on the size of $i$ relative to $b$. 

First suppose that $1\leq i<b$.
Note when $i < b$, $M_i \cong E_{b+1} \mix \G(b-i,t)$, after adding $i$ to the edge labels of $\G(b-i,t)$.

If $b$ and $i$ have the same parity, then $\G(b-i,t) \cong 2 \wr S_{b+t+1-i}$ by \cref{thm:boxes}. Thus in this case, $M_i \cong E_{b+1} \mix (2 \wr S_{b+t+1-i})$, and we already established that $I_i$ is a subgroup of this. On the other hand, if $b$ and $i$ have opposite parity, then $M_i \cong E_{b+1} \mix (2 \wr S_{b+t+1-i})^+$. In this case, $R_i$ stabilizes each of the first $i$ blocks setwise, and the rest of the block structure forms an even permutation group. So $R_i$ acts evenly on the points other than the first $i$ blocks, and so $I_i$ does as well. Thus $I_i \leq E_{b+1} \mix (2 \wr S_{b+t+1-i})^+ = M_i$.

If $i = b$, then again we have $M_b \cong E_{b+1} \mix (2 \wr S_{t+1})$, and $I_i$ was already demonstrated to lie in this subgroup.

If $i = b+1$, then $M_i$ has a structure that follows from the same logic as the $j = r-3$ case of the proof of \cref{glueXYZ}. In particular, $M_i \cong E_{b+1} \mix S_{t}$. On the other hand, it is clear that $R_i$ (and thus $I_i$) has no action within the blocks other than the first $i$, and so $I_i \leq E_{b+1} \mix S_{t} = M_i$.

Finally, if $b+1 < i \leq m-2$, then $M_i$ fixes the first $i$ blocks pointwise, and $M_i \cong S_{b+t+1-i}$, permuting the remaining blocks. Clearly $R_i$ (and thus $I_i$) must also fix the first $i$ blocks pointwise and permute the remaining blocks with no action within blocks, so $I_i \leq M_i$.

\end{proof}

\begin{theorem}\label{thm:family10}
For $b_i \geq 1$ and $t_i \geq 2$, if $b_1$ and $b_2$ are odd, then the group $\G(b_1,t_1,b_2,t_2)$ is a string C-group isomorphic to $A_{2(b_1+b_2+t_1+t_2) +3}$.
\end{theorem}

\begin{proof}
By \cref{glue-box-left}, the first condition of \cref{string-C} holds.  By
duality, the second condition also holds.  Therefore $\G(b_1,t_1,b_2,t_2)$
is a string C-group.  It is isomorphic to $A_n$ where $n=2(b_1+b_2+t_1+t_2)+3$ by \cref{glued-primitive}.
\end{proof}

It is worth noting that $\G = \G(b_1,t_1,b_2,t_2)$ is not a string C-group when both $b_1$ and $b_2$ are even.  As shown in \cref{prop:notC}, $\G_0 \cap \G_{r-1}$ is larger than $\G_{0,r-1}$ in these cases.

\begin{corollary}
\label{cor:more-examples}
Let $b \geq 1$ be odd and $n \equiv 3 \pmod{4}$ with $n \geq 4b+11$. There is a self-dual string C-group $\G(b, t, b, t)$ of rank $\frac{n-3}{2}$ isomorphic to $A_n$ where $t = \frac{n-4b-3}{4}$.
\end{corollary}

Note that when $b\geq 3$, the string C-group $\G(b,t,b,t)$ is not isomorphic to any member of the three self-dual families: $XX(m,m)$, $YY(m,m)$, and $ZZ(m,m)$.  This can be seen from the type of $\G(b,t,b,t)$ which is 
$\{3^b,6^3,3^b\}$ when $t=2$ and $\{3^b,6^2,3^{t-3},6,3^{t-3},6^2,3^b\}$ otherwise.

\section{Conclusion and Open Problems}

When $n=4m+3$ with $m\geq 3$, \cref{cor:3examples} gives three self-dual string C-groups of rank $\frac{n-3}{2}$ isomorphic to $A_n$, and \cref{cor:more-examples} gives $\left \lfloor \frac{n-15}{8} \right\rfloor$ more (one for each odd number $b$ satisfying $3 \leq b \leq m-2$). This gives us $\left\lfloor \frac{n+9}{8}\right\rfloor$ non-isomorphic self-dual string C-groups of rank $\frac{n-3}{2}$ isomorphic to $A_n$, proving \cref{thm:real-main}.  
For example, we construct five
self-dual string C-group representations of $A_{31}$ of rank $14$: $
XX(7,7)$, $YY(7,7)$, $ZZ(7,7)$, $\G(3,4,3,4),$ and $\G(5,2,5,2).$

\medskip
\noindent
\textbf{Open Problem 1.}
Find additional families of self-dual string C-groups of rank $
\left\lfloor \frac{n-1}{2}\right\rfloor-1$ 
isomorphic to $A_n$ when $n\equiv 3\pmod 4$, or prove that the
representations constructed in this paper complete the classification.

The constructions above concern degrees congruent to $3$ modulo $4$.
The computational data suggest that another family could exist for
a different congruence class.  As shown in Appendix~\ref{appendix:A16},
there are self-dual string C-groups of rank $
\left\lfloor \frac{n-1}{2}\right\rfloor-1$ 
isomorphic to $A_n$ when $n=16$.  We believe that the first example
in Appendix~\ref{appendix:A16} belongs to an infinite family occurring
whenever $n\equiv 0\pmod 8$.

\medskip
\noindent
\textbf{Conjecture 1.}
For every $n\geq 16$ with $n\equiv 0\pmod 8$, there exists a
self-dual string C-group of rank $
\left\lfloor \frac{n-1}{2}\right\rfloor-1 $
isomorphic to $A_n$.

Finally, among the alternating groups $A_n$ with $n\leq 19$, our
computations find only five self-dual string C-groups attaining the
maximum possible rank of a string C-group representation.  All five
occur when $n\leq 12$; see \cref{tab:self-dual-an}.  Thus, the
representations constructed in this paper are especially significant
if their rank is the highest possible for a self-dual string C-group
representation of the corresponding alternating group.

\medskip
\noindent
\textbf{Open Problem 2.}
Prove or disprove that there is no self-dual string C-group of rank 
$\left\lfloor \frac{n-1}{2}\right\rfloor$
isomorphic to $A_n$ when $n>12$.

\section*{Acknowledgments}
The calculations in this paper were independently checked in GAP \cite{GAP} and Magma \cite{Magma}.

\appendix

\section{Self-dual string C-groups of highest rank for small $n$}  

In this appendix, we give the type (Schl\"afli symbol), generators, and permutation representation graphs of the self-dual string C-groups of highest rank for $A_n$ with $n \in \{10, 12, 13, 14, 16, 17, 18\}$ (see \cref{tab:self-dual-an}).

\subsection{$A_{10}$} \label{appendix:A10}

\small
\setlength{\tabcolsep}{4pt}
\renewcommand{\arraystretch}{1.35}

\begin{longtable}{@{} l p{0.25\textwidth}p{0.65\textwidth}@{}}
\toprule
\textbf{Type} & \textbf{Generators} & \textbf{Graph} \\
\midrule
\endfirsthead

\toprule
\textbf{Type} & \textbf{Generators} & \textbf{Graph} \\
\midrule
\endhead

$\{9,9\}$ &
\repcell{
\rho_0 &= (3,4)(5,6)(7,8)(9,10)\\
\rho_1 &= (2,3)(4,5)(6,7)(8,9)\\
\rho_2 &= (1,2)(3,4)(5,6)(7,8)
}
&
\graphcell{
*+[o][F]{1} \ar@{-}[r]^{2} & *+[o][F]{2} \ar@{-}[r]^{1} & *+[o][F]{3} \ar@<+.5ex>@{-}[r]^{2} \ar@<-.5ex>@{-}[r]_{0} & *+[o][F]{4} \ar@{-}[r]^{1} & *+[o][F]{5} \ar@<+.5ex>@{-}[r]^{2} \ar@<-.5ex>@{-}[r]_{0} & *+[o][F]{6} \ar@{-}[r]^{1} & *+[o][F]{7} \ar@<+.5ex>@{-}[r]^{2} \ar@<-.5ex>@{-}[r]_{0} & *+[o][F]{8} \ar@{-}[r]^{1} & *+[o][F]{9} \ar@{-}[r]^{0} & *+[o][F]{10}
}
\\
\midrule

$\{9,9\}$ &
\repcell{
\rho_0 &= (3,4)(5,7)(6,8)(9,10)\\
\rho_1 &= (2,3)(4,6)(5,8)(7,9)\\
\rho_2 &= (1,2)(3,5)(4,7)(6,8)
}
&
\graphcell{
& & & *+[o][F]{4} \ar@{-}[d]^{1} \ar@{-}[drr]^{2} & & & & \\
*+[o][F]{1} \ar@{-}[r]^{2} & *+[o][F]{2} \ar@{-}[r]^{1} & *+[o][F]{3} \ar@{-}[ur]^{0} \ar@{-}[drr]_{2} & *+[o][F]{6} \ar@<+.5ex>@{-}[r]^{0} \ar@<-.5ex>@{-}[r]_{2} & *+[o][F]{8}  \ar@{-}[d]^{1} & *+[o][F]{7} \ar@{-}[r]^{1} & *+[o][F]{9} \ar@{-}[r]^{0} & *+[o][F]{10} \\
& & & & *+[o][F]{5} \ar@{-}[ur]_{0} & & &
}
\\
\midrule

$\{21,21\}$ &
\repcell{
\rho_0 &= (2,3)(4,5)(6,7)(8,9)\\
\rho_1 &= (1,2)(3,4)(5,6)(9,10)\\
\rho_2 &= (2,3)(4,5)(6,8)(7,9)
}
&
\graphcell{
& & & & & & *+[o][F]{7} \ar@{-}[dr]^{2} & \\
*+[o][F]{1} \ar@{-}[r]^{1} & *+[o][F]{2} \ar@<+.5ex>@{-}[r]^{2} \ar@<-.5ex>@{-}[r]_{0} & *+[o][F]{3} \ar@{-}[r]^{1} & *+[o][F]{4} \ar@<+.5ex>@{-}[r]^{2} \ar@<-.5ex>@{-}[r]_{0} & *+[o][F]{5} \ar@{-}[r]^{1} & *+[o][F]{6} \ar@{-}[ur]^{0} \ar@{-}[dr]_{2} & & *+[o][F]{9} \ar@{-}[r]^{1} & *+[o][F]{10} \\
& & & & & & *+[o][F]{8} \ar@{-}[ur]_{0} &
}
\\
\midrule

$\{21,21\}$ &
\repcell{
\rho_0 &= (2,3)(4,5)(6,7)(8,9)\\
\rho_1 &= (1,2)(5,6)(7,8)(9,10)\\
\rho_2 &= (2,4)(3,5)(6,8)(7,9)
}
&
\graphcell{
& & *+[o][F]{3} \ar@{-}[dr]^{2} & & & *+[o][F]{7} \ar@{-}[dd]^{1} \ar@{-}[dr]^{2} & & \\
*+[o][F]{1} \ar@{-}[r]^{1} & *+[o][F]{2} \ar@{-}[ur]^{0} \ar@{-}[dr]_{2} & & *+[o][F]{5} \ar@{-}[r]^{1} & *+[o][F]{6} \ar@{-}[ur]_{0} \ar@{-}[dr]_{2} & & *+[o][F]{9} \ar@{-}[r]^{1} & *+[o][F]{10} \\
& & *+[o][F]{4} \ar@{-}[ur]_{0} & & & *+[o][F]{8} \ar@{-}[ur]_{0} & &
}
\\
\midrule

$\{21,21\}$ &
\repcell{
\rho_0 &= (2,3)(4,5)(6,8)(7,9)\\
\rho_1 &= (1,2)(4,6)(5,7)(9,10)\\
\rho_2 &= (2,4)(3,5)(6,7)(8,9)
}
&
\graphcell{
& *+[o][F]{3} \ar@{-}[r]^{2} & *+[o][F]{5} \ar@{-}[r]^{1} & *+[o][F]{7} \ar@{-}[r]^{0} & *+[o][F]{9} \ar@{-}[r]^{1} & *+[o][F]{10} \\
*+[o][F]{1} \ar@{-}[r]^{1} & *+[o][F]{2} \ar@{-}[u]^{0} \ar@{-}[r]^{2} & *+[o][F]{4} \ar@{-}[u]^{0} \ar@{-}[r]^{1} & *+[o][F]{6} \ar@{-}[u]^{2} \ar@{-}[r]^{0} & *+[o][F]{8} \ar@{-}[u]^{2} &
}
\\
\midrule

$\{21,21\}$ &
\repcell{
\rho_0 &= (2,3)(4,5)(6,8)(7,9)\\
\rho_1 &= (1,2)(4,6)(5,7)(8,10)\\
\rho_2 &= (2,4)(3,5)(6,7)(8,9)
}
&
\graphcell{
& *+[o][F]{3} \ar@{-}[r]^{2} & *+[o][F]{5} \ar@{-}[r]^{1} & *+[o][F]{7} \ar@{-}[r]^{0} & *+[o][F]{9} & \\
*+[o][F]{1} \ar@{-}[r]^{1} & *+[o][F]{2} \ar@{-}[u]^{0} \ar@{-}[r]^{2} & *+[o][F]{4} \ar@{-}[u]^{0} \ar@{-}[r]^{1} & *+[o][F]{6} \ar@{-}[u]^{2} \ar@{-}[r]^{0} & *+[o][F]{8} \ar@{-}[u]^{2} \ar@{-}[r]^{1} & *+[o][F]{10}
}
\\
\midrule

$\{21,21\}$ &
\repcell{
\rho_0 &= (1,2)(3,5)(6,7)(8,9)\\
\rho_1 &= (2,3)(4,6)(5,7)(9,10)\\
\rho_2 &= (1,2)(3,5)(6,8)(7,9)
}
&
\graphcell{
& & & & & *+[o][F]{6} \ar@{-}[r]^{1} & *+[o][F]{4} \\
*+[o][F]{1} \ar@<+.5ex>@{-}[r]^{0} \ar@<-.5ex>@{-}[r]_{2} & *+[o][F]{2} \ar@{-}[r]^{1} & *+[o][F]{3} \ar@<+.5ex>@{-}[r]^{0} \ar@<-.5ex>@{-}[r]_{2} & *+[o][F]{5} \ar@{-}[r]^{1} & *+[o][F]{7} \ar@{-}[ur]^{0} \ar@{-}[dr]_{2} & & *+[o][F]{8} \ar@{-}[ul]_{2} \ar@{-}[dl]^{0} \\
& & & & & *+[o][F]{9} \ar@{-}[r]^{1} & *+[o][F]{10}
}
\\
\midrule

$\{21,21\}$ &
\repcell{
\rho_0 &= (2,3)(4,5)(6,8)(9,10)\\
\rho_1 &= (1,2)(4,6)(5,7)(8,9)\\
\rho_2 &= (2,4)(3,5)(6,9)(8,10)
}
&
\graphcell{
*+[o][F]{1} \ar@{-}[r]^{1} & *+[o][F]{2} \ar@{-}[dl]^{0} \ar@{-}[dr]_{2} & & & *+[o][F]{8} \ar@{-}[dr]^{2} & \\
*+[o][F]{3} \ar@{-}[dr]_{2} & & *+[o][F]{4} \ar@{-}[r]^{1} & *+[o][F]{6} \ar@{-}[ur]^{0} \ar@{-}[dr]_{2} & & *+[o][F]{10} \\
*+[o][F]{7} \ar@{-}[r]^{1} & *+[o][F]{5} \ar@{-}[ur]_{0} & & & *+[o][F]{9} \ar@<+.5ex>@{-}[uu]^{1} \ar@<-.5ex>@{-}[ur]_{0} &
}
\\
\midrule

$\{21,21\}$ &
\repcell{
\rho_0 &= (1,2)(3,5)(4,6)(7,9)\\
\rho_1 &= (2,3)(5,7)(6,8)(9,10)\\
\rho_2 &= (1,2)(3,6)(4,5)(8,10)
}
&
\graphcell{
& & & *+[o][F]{6} \ar@{-}[r]^{1} & *+[o][F]{8} \ar@{-}[r]^{2} & *+[o][F]{10} \ar@{-}[dd]^{1} \\
*+[o][F]{1} \ar@<+.5ex>@{-}[r]^{0} \ar@<-.5ex>@{-}[r]_{2} & *+[o][F]{2} \ar@{-}[r]^{1} & *+[o][F]{3} \ar@{-}[ur]^{2} \ar@{-}[dr]_{0} & & *+[o][F]{4} \ar@{-}[ul]_{0} \ar@{-}[dl]^{2} & \\
& & & *+[o][F]{5} \ar@{-}[r]^{1} & *+[o][F]{7} \ar@{-}[r]^{0} & *+[o][F]{9}
}
\\
\midrule

$\{21,21\}$ &
\repcell{
\rho_0 &= (1,2)(3,4)(5,6)(9,10)\\
\rho_1 &= (2,3)(4,5)(6,7)(8,9)\\
\rho_2 &= (1,2)(5,6)(7,8)(9,10)
}
&
\graphcell{
*+[o][F]{1} \ar@<+.5ex>@{-}[r]^{0} \ar@<-.5ex>@{-}[r]_{2} & *+[o][F]{2} \ar@{-}[r]^{1} & *+[o][F]{3} \ar@{-}[r]^{0} & *+[o][F]{4} \ar@{-}[r]^{1} & *+[o][F]{5} \ar@<+.5ex>@{-}[r]^{0} \ar@<-.5ex>@{-}[r]_{2} & *+[o][F]{6} \ar@{-}[r]^{1} & *+[o][F]{7} \ar@{-}[r]^{2} & *+[o][F]{8} \ar@{-}[r]^{1} & *+[o][F]{9} \ar@<+.5ex>@{-}[r]^{0} \ar@<-.5ex>@{-}[r]_{2} & *+[o][F]{10}
}
\\
\midrule

$\{21,21\}$ &
\repcell{
\rho_0 &= (1,2)(3,4)(5,10)(6,7)\\
\rho_1 &= (2,3)(4,5)(7,8)(9,10)\\
\rho_2 &= (1,7)(2,6)(5,10)(8,9)
}
&
\graphcell{
& *+[o][F]{2} \ar@{-}[r]^{1} \ar@{-}[dl] & *+[o][F]{3} \ar@{-}[r]^{0} & *+[o][F]{4} \ar@{-}[r]^{1} & *+[o][F]{5} \\
*+[o][F]{1} \ar@{-}[ur]^{0} \ar@{-}[dr]_{2} & & *+[o][F]{6} \ar@{-}[ul]_{2} & & \\
& *+[o][F]{7} \ar@{-}[ur]_{0} \ar@{-}[r]^{1} & *+[o][F]{8} \ar@{-}[r]^{2} & *+[o][F]{9} \ar@{-}[r]^{1} & *+[o][F]{10} \ar@<+.5ex>@{-}[uu]^{0} \ar@<-.5ex>@{-}[uu]_{2}
}
\\

\bottomrule
\end{longtable}

\subsection{$A_{12}$} \label{appendix:A12}

Note that the graph below repeats vertices 6 and 8 for readability. The duality is realized by a reflection in the graph through the line that connects vertices 3 and 4. 

We note that it is much more difficult to draw this graph than all the others in the appendix because it is the only graph where no generator fixes any vertex. \\

\small
\setlength{\tabcolsep}{4pt}
\renewcommand{\arraystretch}{1.35}

\begin{longtable}{@{} p{0.12\textwidth} p{0.40\textwidth}p{0.45\textwidth}@{}}
\toprule
\textbf{Type} & \textbf{Generators} & \textbf{Graph} \\
\midrule
\endfirsthead

\toprule
\textbf{Type} & \textbf{Generators} & \textbf{Graph} \\
\midrule
\endhead

$\{3,6,6,3\}$ & 
\repcell{
\rho_0 &= (1,2)(3,9)(4,7)(5,8)(6,11)(10,12)\\
\rho_1 &= (1,3)(2,6)(4,10)(5,7)(8,12)(9,11)\\
\rho_2 &= (1,2)(3,4)(5,8)(6,12)(7,9)(10,11)\\
\rho_3 &= (1,4)(2,7)(3,10)(5,6)(8,11)(9,12)\\
\rho_4 &= (1,5)(2,8)(3,7)(4,9)(6,12)(10,11)
}
&
\graphcell{
*+[o][F]{6} \ar@{-}[r]^1 \ar@{-}[ddddr]^(0.2)0
& *+[o][F]{2} \ar@{-}[ddddl]_(0.2)4 \ar@{-}[r]^0 \ar@<-.5ex>@{-}[r]_2 \ar@{-}[dr]_(0.3)3
& *+[o][F]{1} \ar@{-}[r]^4 \ar@{-}[ddl]_(0.7)1 \ar@{-}[ddr]^(0.7)3
& *+[o][F]{5} \ar@{-}[ddddr]^(0.2)3 \ar@{-}[r]^0 \ar@<-.5ex>@{-}[r]_2 \ar@{-}[dl]^(0.3)1
& *+[o][F]{8} \ar@{-}[ddddl]_(0.2)1
\\
&& *+[o][F]{7} \ar@{-}[dd]_(0.2)2 \ar@{-}[dl]^4 \ar@{-}[dr]^(0.2)0 \\
& *+[o][F]{3} \ar@{-}[rr]_(0.4)2 \ar@{-}[dr]^0 && *+[o][F]{4} \ar@{-}[dl]^(0.2)4 \\
&& *+[o][F]{9} \ar@{-}[dl]^(0.2)1 \ar@{-}[dr]^(0.7)3 \\
*+[o][F]{8} \ar@{-}[r]_3
& *+[o][F]{11} \ar@{-}[r]^2 \ar@<-.5ex>@{-}[r]_4
& *+[o][F]{10} \ar@{-}[uul]^3 \ar@{-}[uur]^1 \ar@{-}[r]_0
& *+[o][F]{12} \ar@{-}[r]^2 \ar@<-.5ex>@{-}[r]_4
& *+[o][F]{6}
}

\\

\bottomrule
\end{longtable}

\subsection{$A_{13}$} \label{appendix:A13}

\small
\setlength{\tabcolsep}{4pt}
\renewcommand{\arraystretch}{1.35}

\begin{longtable}{@{}p{0.15\textwidth} p{0.25\textwidth}p{0.58\textwidth}@{}}
\toprule
\textbf{Type} & \textbf{Generators} & \textbf{Graph} \\
\midrule
\endfirsthead

\toprule
\textbf{Type} & \textbf{Generators} & \textbf{Graph} \\
\midrule
\endhead

$\{4,6,4\}$ & 
\repcell{
\rho_0 &= (5,6)(11,12)\\
\rho_1 &= (4,5)(6,7)(10,11)(12,13)\\
\rho_2 &= (1,2)(3,4)(7,8)(9,10)\\
\rho_3 &= (2,3)(8,9)
}
&
\graphcell{
*+[o][F]{1} \ar@{-}[r]^{2} &
*+[o][F]{2} \ar@{-}[r]^{3} &
*+[o][F]{3} \ar@{-}[r]^{2} &
*+[o][F]{4} \ar@{-}[r]^{1} &
*+[o][F]{5} \ar@{-}[r]^{0} &
*+[o][F]{6} \ar@{-}[dr]^{1} & \\
& & & & & & *+[o][F]{7} \\
*+[o][F]{13} \ar@{-}[r]^{1} &
*+[o][F]{12} \ar@{-}[r]^{0} &
*+[o][F]{11} \ar@{-}[r]^{1} &
*+[o][F]{10} \ar@{-}[r]^{2} &
*+[o][F]{9} \ar@{-}[r]^{3} &
*+[o][F]{8} \ar@{-}[ur]_{2} & 
}
\\
\midrule

$\{6,6,6\}$ & 
\repcell{
\rho_0 &= (3,4)(5,6)\\
\rho_1 &= (2,3)(4,5)(6,7)(12,13)\\
\rho_2 &= (1,2)(7,8)(9,10)(11,12)\\
\rho_3 &= (8,9)(10,11)
}
&
\graphcell{
*+[o][F]{1} \ar@{-}[r]^{2} &
*+[o][F]{2} \ar@{-}[r]^{1} &
*+[o][F]{3} \ar@{-}[r]^{0} &
*+[o][F]{4} \ar@{-}[r]^{1} &
*+[o][F]{5} \ar@{-}[r]^{0} &
*+[o][F]{6} \ar@{-}[dr]^{1} & \\
& & & & & & *+[o][F]{7} \\
*+[o][F]{13} \ar@{-}[r]^{1} &
*+[o][F]{12} \ar@{-}[r]^{2} &
*+[o][F]{11} \ar@{-}[r]^{3} &
*+[o][F]{10} \ar@{-}[r]^{2} &
*+[o][F]{9} \ar@{-}[r]^{3} &
*+[o][F]{8} \ar@{-}[ur]^{2} & }
\\

\bottomrule
\end{longtable}

\subsection{$A_{14}$} \label{appendix:A14}

\small
\setlength{\tabcolsep}{4pt}
\renewcommand{\arraystretch}{1.35}

\begin{longtable}{@{} p{0.12\textwidth} p{0.24\textwidth}p{0.62\textwidth}@{}}
\toprule
\textbf{Type} & \textbf{Generators} & \textbf{Graph} \\
\midrule
\endfirsthead

\toprule
\textbf{Type} & \textbf{Generators} & \textbf{Graph} \\
\midrule
\endhead

$\{6,10,6\}$ & 
\repcell{
\rho_0 &= (1,2)(6,7)\\
\rho_1 &= (2,3)(4,13)(5,14)(7,8)\\ &\quad (9,10)(11,12)\\
\rho_2 &= (1,6)(2,7)(3,4)(8,9)\\ &\quad (10,11)(12,13)\\
\rho_3 &= (4,5)(13,14)
}
&
\graphcell[@-0pc@C=1.5em]{
*+[o][F]{1} \ar@{-}[r]^{0} \ar@{-}[d]^{2}
& *+[o][F]{2} \ar@{-}[rrr]^{1} \ar@{-}[d]^{2}
& & &
*+[o][F]{3} \ar@{-}[rrr]^{2}
& & &
*+[o][F]{4} \ar@{-}[r]^{3} \ar@{-}[d]^{1}
& *+[o][F]{5} \ar@{-}[d]^{1} \\
*+[o][F]{6} \ar@{-}[r]_{0}
& *+[o][F]{7} \ar@{-}[r]_{1}
& *+[o][F]{8} \ar@{-}[r]_{2}
& *+[o][F]{9} \ar@{-}[r]_{1}
& *+[o][F]{10} \ar@{-}[r]_{2}
& *+[o][F]{11} \ar@{-}[r]_{1}
& *+[o][F]{12} \ar@{-}[r]_{2}
& *+[o][F]{13} \ar@{-}[r]_{3}
& *+[o][F]{14}
}
\\

\bottomrule
\end{longtable}

\subsection{$A_{16}$} \label{appendix:A16}

\small
\setlength{\tabcolsep}{4pt}
\renewcommand{\arraystretch}{1.35}

\begin{longtable}{@{} p{0.12\textwidth} p{0.29\textwidth}p{0.57\textwidth}@{}}
\toprule
\textbf{Type} & \textbf{Generators} & \textbf{Graph} \\
\midrule
\endfirsthead

\toprule
\textbf{Type} & \textbf{Generators} & \textbf{Graph} \\
\midrule
\endhead

$\{3,6,4,6,3\}$ & 
\repcell{
\rho_0 &= (1,2)(8,9)\\
\rho_1 &= (2,3)(9,10)\\
\rho_2 &= (3,4)(5,14)(6,15)(7,16)\\
       &\quad (10,11)(12,13)\\
\rho_3 &= (1,8)(2,9)(3,10)(4,5)\\
       &\quad (11,12)(13,14)\\
\rho_4 &= (5,6)(14,15)\\
\rho_5 &= (6,7)(15,16)
}
&
\graphcell[@-0pc@C=1.5em]{
*+[o][F]{1} \ar@{-}[r]^{0} \ar@{-}[d]^{3}
& *+[o][F]{2} \ar@{-}[r]^{1} \ar@{-}[d]^{3}
& *+[o][F]{3} \ar@{-}[rr]^{2} \ar@{-}[d]^{3}
& &
*+[o][F]{4} \ar@{-}[rr]^{3}
& &
*+[o][F]{5} \ar@{-}[r]^{4} \ar@{-}[d]^{2}
& *+[o][F]{6} \ar@{-}[r]^{5} \ar@{-}[d]^{2}
& *+[o][F]{7} \ar@{-}[d]^{2} \\
*+[o][F]{8} \ar@{-}[r]_{0}
& *+[o][F]{9} \ar@{-}[r]_{1}
& *+[o][F]{10} \ar@{-}[r]_{2}
& *+[o][F]{11} \ar@{-}[r]_{3}
& *+[o][F]{12} \ar@{-}[r]_{2}
& *+[o][F]{13} \ar@{-}[r]_{3}
& *+[o][F]{14} \ar@{-}[r]_{4}
& *+[o][F]{15} \ar@{-}[r]_{5}
& *+[o][F]{16}
}
\\
\midrule

$\{4,4,4,4,4\}$ & 
\repcell{
\rho_0 &= (1,8)(2,9)\\
\rho_1 &= (2,3)(9,10)\\
\rho_2 &= (1,2)(3,4)(5,14)(6,15)\\
       &\quad (7,16)(8,9)(10,11)(12,13)\\
\rho_3 &= (1,8)(2,9)(3,10)(4,5)\\
       &\quad (6,7)(11,12)(13,14)(15,16)\\
\rho_4 &= (5,6)(14,15)\\
\rho_5 &= (6,15)(7,16)
}
&
\graphcell[@-0pc@C=1.5em]{
*+[o][F]{1} \ar@{-}[r]^{2} \ar@<+.5ex>@{-}[d]^{0} \ar@<-.5ex>@{-}[d]_{3}
& *+[o][F]{2} \ar@{-}[r]^{1} \ar@<+.5ex>@{-}[d]^{0} \ar@<-.5ex>@{-}[d]_{3}
& *+[o][F]{3} \ar@{-}[rr]^{2} \ar@{-}[d]^{3}
& &
*+[o][F]{4} \ar@{-}[rr]^{3}
& &
*+[o][F]{5} \ar@{-}[r]^{4} \ar@{-}[d]^{2}
& *+[o][F]{6} \ar@{-}[r]^{3} \ar@<+.5ex>@{-}[d]^{2} \ar@<-.5ex>@{-}[d]_{5}
& *+[o][F]{7} \ar@<+.5ex>@{-}[d]^{2} \ar@<-.5ex>@{-}[d]_{5} \\
*+[o][F]{8} \ar@{-}[r]_{2}
& *+[o][F]{9} \ar@{-}[r]_{1}
& *+[o][F]{10} \ar@{-}[r]_{2}
& *+[o][F]{11} \ar@{-}[r]_{3}
& *+[o][F]{12} \ar@{-}[r]_{2}
& *+[o][F]{13} \ar@{-}[r]_{3}
& *+[o][F]{14} \ar@{-}[r]_{4}
& *+[o][F]{15} \ar@{-}[r]_{3}
& *+[o][F]{16}
}
\\
\midrule

$\{3,4,4,4,3\}$ & 
\repcell{
\rho_0 &= (1,2)(8,9)\\
\rho_1 &= (2,3)(9,10)\\
\rho_2 &= (1,2)(3,4)(5,14)(6,15)\\
       &\quad (7,16)(8,9)(10,11)(12,13)\\
\rho_3 &= (1,8)(2,9)(3,10)(4,5)\\
       &\quad (6,7)(11,12)(13,14)(15,16)\\
\rho_4 &= (5,6)(14,15)\\
\rho_5 &= (6,7)(15,16)
}
&
\graphcell[@-0pc@C=1.5em]{
*+[o][F]{1} \ar@<+.5ex>@{-}[r]^{0} \ar@<-.5ex>@{-}[r]_{2} \ar@{-}[d]^{3}
& *+[o][F]{2} \ar@{-}[r]^{1} \ar@{-}[d]^{3}
& *+[o][F]{3} \ar@{-}[rr]^{2} \ar@{-}[d]^{3}
& &
*+[o][F]{4} \ar@{-}[rr]^{3}
& &
*+[o][F]{5} \ar@{-}[r]^{4} \ar@{-}[d]^{2}
& *+[o][F]{6} \ar@<+.5ex>@{-}[r]^{5} \ar@<-.5ex>@{-}[r]_{3} \ar@{-}[d]^{2}
& *+[o][F]{7} \ar@{-}[d]^{2} \\
*+[o][F]{8} \ar@<+.5ex>@{-}[r]^{0} \ar@<-.5ex>@{-}[r]_{2}
& *+[o][F]{9} \ar@{-}[r]_{1}
& *+[o][F]{10} \ar@{-}[r]_{2}
& *+[o][F]{11} \ar@{-}[r]_{3}
& *+[o][F]{12} \ar@{-}[r]_{2}
& *+[o][F]{13} \ar@{-}[r]_{3}
& *+[o][F]{14} \ar@{-}[r]_{4}
& *+[o][F]{15} \ar@<+.5ex>@{-}[r]^{5} \ar@<-.5ex>@{-}[r]_{3}
& *+[o][F]{16}
}
\\

\bottomrule
\end{longtable}

\subsection{$A_{17}$} \label{appendix:A17}

\small
\setlength{\tabcolsep}{4pt}
\renewcommand{\arraystretch}{1.35}

\begin{longtable}{@{} p{0.12\textwidth} p{0.29\textwidth}p{0.57\textwidth}@{}}
\toprule
\textbf{Type} & \textbf{Generators} & \textbf{Graph} \\
\midrule
\endfirsthead

\toprule
\textbf{Type} & \textbf{Generators} & \textbf{Graph} \\
\midrule
\endhead

$\{4,12,6,12,4\}$ & 
\repcell{
\rho_0 &= (2,3)(6,7)\\
\rho_1 &= (1,2)(3,4)(5,6)(7,8)\\
\rho_2 &= (4,5)(8,9)\\
\rho_3 &= (9,10)(13,14)\\
\rho_4 &= (10,11)(12,13)(14,15)(16,17)\\
\rho_5 &= (11,12)(15,16)
}
&
\graphcell{
*+[o][F]{1} \ar@{-}[r]^{1}
& *+[o][F]{2} \ar@{-}[r]^{0}
& *+[o][F]{3} \ar@{-}[r]^{1}
& *+[o][F]{4} \ar@{-}[r]^{2}
& *+[o][F]{5} \ar@{-}[r]^{1}
& *+[o][F]{6} \ar@{-}[r]^{0}
& *+[o][F]{7} \ar@{-}[r]^{1}
& *+[o][F]{8} \ar@{-}[dr]^{2} & \\
&&&&&&&& *+[o][F]{9} \\
*+[o][F]{17} \ar@{-}[r]^{4}
& *+[o][F]{16} \ar@{-}[r]^{5}
& *+[o][F]{15} \ar@{-}[r]^{4}
& *+[o][F]{14} \ar@{-}[r]^{3}
& *+[o][F]{13} \ar@{-}[r]^{4}
& *+[o][F]{12} \ar@{-}[r]^{5}
& *+[o][F]{11} \ar@{-}[r]^{4}
& *+[o][F]{10} \ar@{-}[ur]^{3}
}
\\
\midrule

$\{6,12,6,12,6\}$ & 
\repcell{
\rho_0 &= (4,5)(6,7)\\
\rho_1 &= (1,2)(3,4)(5,6)(7,8)\\
\rho_2 &= (2,3)(8,9)\\
\rho_3 &= (9,10)(15,16)\\
\rho_4 &= (10,11)(12,13)(14,15)(16,17)\\
\rho_5 &= (11,12)(13,14)
}
&
\graphcell{
*+[o][F]{1} \ar@{-}[r]^{1}
& *+[o][F]{2} \ar@{-}[r]^{2}
& *+[o][F]{3} \ar@{-}[r]^{1}
& *+[o][F]{4} \ar@{-}[r]^{0}
& *+[o][F]{5} \ar@{-}[r]^{1}
& *+[o][F]{6} \ar@{-}[r]^{0}
& *+[o][F]{7} \ar@{-}[r]^{1}
& *+[o][F]{8} \ar@{-}[dr]^{2} & \\
&&&&&&&& *+[o][F]{9} \\
*+[o][F]{17} \ar@{-}[r]^{4}
& *+[o][F]{16} \ar@{-}[r]^{3}
& *+[o][F]{15} \ar@{-}[r]^{4}
& *+[o][F]{14} \ar@{-}[r]^{5}
& *+[o][F]{13} \ar@{-}[r]^{4}
& *+[o][F]{12} \ar@{-}[r]^{5}
& *+[o][F]{11} \ar@{-}[r]^{4}
& *+[o][F]{10} \ar@{-}[ur]^{3} & 
}
\\

\bottomrule
\end{longtable}

\subsection{$A_{18}$} \label{appendix:A18}

\small
\setlength{\tabcolsep}{4pt}
\renewcommand{\arraystretch}{1.35}

\begin{longtable}{@{} p{0.1\textwidth} p{0.29\textwidth}p{0.6\textwidth}@{}}
\toprule
\textbf{Type} & \textbf{Generators} & \textbf{Graph} \\
\midrule
\endfirsthead

\toprule
\textbf{Type} & \textbf{Generators} & \textbf{Graph} \\
\midrule
\endhead

$\{12,6,12\}$ & 
\repcell{
\rho_0 &= (1,2)(10,11)\\
\rho_1 &= (2,3)(4,13)(5,14)(7,8)(9,10)\\
       &\quad (11,12)(15,16)(17,18)\\
\rho_2 &= (1,10)(2,11)(3,4)(6,7)\\
       &\quad (8,9)(12,13)(14,15)(16,17)\\
\rho_3 &= (4,5)(13,14)
}
&
\graphcell[@-0pc@C=1.4em]{
& & & &
*+[o][F]{1} \ar@{-}[r]^{0} \ar@{-}[d]^{2}
& *+[o][F]{2} \ar@{-}[r]^{1} \ar@{-}[d]^{2}
& *+[o][F]{3} \ar@{-}[r]^{2}
& *+[o][F]{4} \ar@{-}[r]^{3} \ar@{-}[d]^{1}
& *+[o][F]{5} \ar@{-}[d]^{1}
& & & & \\
*+[o][F]{6} \ar@{-}[r]_{2}
& *+[o][F]{7} \ar@{-}[r]_{1}
& *+[o][F]{8} \ar@{-}[r]_{2}
& *+[o][F]{9} \ar@{-}[r]_{1}
& *+[o][F]{10} \ar@{-}[r]_{0}
& *+[o][F]{11} \ar@{-}[r]_{1}
& *+[o][F]{12} \ar@{-}[r]_{2}
& *+[o][F]{13} \ar@{-}[r]_{3}
& *+[o][F]{14} \ar@{-}[r]_{2}
& *+[o][F]{15} \ar@{-}[r]_{1}
& *+[o][F]{16} \ar@{-}[r]_{2}
& *+[o][F]{17} \ar@{-}[r]_{1}
& *+[o][F]{18}
}
\\
\midrule

$\{4,6,4\}$ & 
\repcell{
\rho_0 &= (2,3)(11,12)\\
\rho_1 &= (1,2)(3,4)(5,14)(6,15)\\
       &\quad (8,9)(10,11)(12,13)(16,17)\\
\rho_2 &= (2,11)(3,12)(4,5)(6,7)\\
       &\quad (9,10)(13,14)(15,16)(17,18)\\
\rho_3 &= (5,6)(14,15)
}
&
\graphcell[@-0pc@C=1.4em]{
& &
*+[o][F]{1} \ar@{-}[r]^{1}
& *+[o][F]{2} \ar@{-}[r]^{0} \ar@{-}[d]^{2}
& *+[o][F]{3} \ar@{-}[r]^{1} \ar@{-}[d]^{2}
& *+[o][F]{4} \ar@{-}[r]^{2}
& *+[o][F]{5} \ar@{-}[r]^{3} \ar@{-}[d]^{1}
& *+[o][F]{6} \ar@{-}[r]^{2} \ar@{-}[d]^{1}
& *+[o][F]{7}
& & \\
*+[o][F]{8} \ar@{-}[r]_{1}
& *+[o][F]{9} \ar@{-}[r]_{2}
& *+[o][F]{10} \ar@{-}[r]_{1}
& *+[o][F]{11} \ar@{-}[r]_{0}
& *+[o][F]{12} \ar@{-}[r]_{1}
& *+[o][F]{13} \ar@{-}[r]_{2}
& *+[o][F]{14} \ar@{-}[r]_{3}
& *+[o][F]{15} \ar@{-}[r]_{2}
& *+[o][F]{16} \ar@{-}[r]_{1}
& *+[o][F]{17} \ar@{-}[r]_{2}
& *+[o][F]{18}
}
\\
\midrule

$\{6,14,6\}$ & 
\repcell{
\rho_0 &= (1,2)(14,15)\\
\rho_1 &= (2,3)(4,5)(6,7)(8,9)\\
       &\quad (10,11)(12,17)(13,18)(15,16)\\
\rho_2 &= (1,14)(2,15)(3,4)(5,6)\\
       &\quad (7,8)(9,10)(11,12)(16,17)\\
\rho_3 &= (12,13)(17,18)
}
&
\graphcell[@-0pc@C=1.4em]{
*+[o][F]{1} \ar@{-}[r]^{0} \ar@{-}[d]^{2}
& *+[o][F]{2} \ar@{-}[r]^{1} \ar@{-}[d]^{2}
& *+[o][F]{3} \ar@{-}[r]^{2}
& *+[o][F]{4} \ar@{-}[r]^{1}
& *+[o][F]{5} \ar@{-}[r]^{2}
& *+[o][F]{6} \ar@{-}[r]^{1}
& *+[o][F]{7} \ar@{-}[r]^{2}
& *+[o][F]{8} \ar@{-}[r]^{1}
& *+[o][F]{9} \ar@{-}[r]^{2}
& *+[o][F]{10} \ar@{-}[r]^{1}
& *+[o][F]{11} \ar@{-}[r]^{2}
& *+[o][F]{12} \ar@{-}[r]^{3} \ar@{-}[d]^{1}
& *+[o][F]{13} \ar@{-}[d]^{1} \\
*+[o][F]{14} \ar@{-}[r]_{0}
& *+[o][F]{15} \ar@{-}[rrrrr]_{1}
& & & & &
*+[o][F]{16} \ar@{-}[rrrrr]_{2}
& & & & &
*+[o][F]{17} \ar@{-}[r]_{3}
& *+[o][F]{18}
}
\\
\midrule

$\{6,14,6\}$ & 
\repcell{
\rho_0 &= (1,2)(12,13)\\
\rho_1 &= (2,3)(4,5)(6,7)(8,9)\\
       &\quad (10,17)(11,18)(13,14)(15,16)\\
\rho_2 &= (1,12)(2,13)(3,4)(5,6)\\
       &\quad (7,8)(9,10)(14,15)(16,17)\\
\rho_3 &= (10,11)(17,18)
}
&
\graphcell[@-0pc@C=1.4em]{
*+[o][F]{1} \ar@{-}[r]^{0} \ar@{-}[d]^{2}
& *+[o][F]{2} \ar@{-}[r]^{1} \ar@{-}[d]^{2}
& *+[o][F]{3} \ar@{-}[r]^{2}
& *+[o][F]{4} \ar@{-}[r]^{1}
& *+[o][F]{5} \ar@{-}[r]^{2}
& *+[o][F]{6} \ar@{-}[r]^{1}
& *+[o][F]{7} \ar@{-}[r]^{2}
& *+[o][F]{8} \ar@{-}[r]^{1}
& *+[o][F]{9} \ar@{-}[r]^{2}
& *+[o][F]{10} \ar@{-}[r]^{3} \ar@{-}[d]^{1}
& *+[o][F]{11} \ar@{-}[d]^{1} \\
*+[o][F]{12} \ar@{-}[r]_{0}
& *+[o][F]{13} \ar@{-}[rr]_{1}
& &
*+[o][F]{14} \ar@{-}[rr]_{2}
& &
*+[o][F]{15} \ar@{-}[rr]_{1}
& &
*+[o][F]{16} \ar@{-}[rr]_{2}
& &
*+[o][F]{17} \ar@{-}[r]_{3}
& *+[o][F]{18}
}
\\

\bottomrule
\end{longtable}

\bibliographystyle{plain}
\bibliography{Alt}

\begin{thebibliography}{10}

\bibitem{Magma}
Wieb Bosma, John Cannon, and Catherine Playoust.
\newblock The {M}agma algebra system. {I}. {T}he user language.
\newblock {\em J. Symbolic Comput.}, 24(3-4):235--265, 1997.
\newblock Computational algebra and number theory (London, 1993).

\bibitem{CFLM2016}
Peter~J. Cameron, Maria~Elisa Fernandes, Dimitri Leemans, and Mark Mixer.
\newblock String {C}-groups as transitive subgroups of {${\rm S}_n$}.
\newblock {\em J. Algebra}, 447:468--478, 2016.

\bibitem{highest-rank-an}
Peter~J. Cameron, Maria~Elisa Fernandes, Dimitri Leemans, and Mark Mixer.
\newblock Highest rank of a polytope for {$A_n$}.
\newblock {\em Proc. Lond. Math. Soc. (3)}, 115(1):135--176, 2017.

\bibitem{var-gps}
Gabe Cunningham.
\newblock Variance groups and the structure of mixed polytopes.
\newblock In {\em Rigidity and symmetry}, volume~70 of {\em Fields Inst.
  Commun.}, pages 97--116. Springer, New York, 2014.

\bibitem{poly-mix}
Gabe Cunningham and Isabel Hubard.
\newblock Polytopality criteria for the mix of polytopes and maniplexes.
\newblock {\em arXiv preprint arXiv:2506.19334}, 2025.

\bibitem{DM96}
John~D. Dixon and Brian Mortimer.
\newblock {\em Permutation groups}, volume 163 of {\em Graduate Texts in
  Mathematics}.
\newblock Springer-Verlag, New York, 1996.

\bibitem{FL11}
Maria~Elisa Fernandes and Dimitri Leemans.
\newblock Polytopes of high rank for the symmetric groups.
\newblock {\em Adv. Math.}, 228(6):3207--3222, 2011.

\bibitem{FLM12}
Maria~Elisa Fernandes, Dimitri Leemans, and Mark Mixer.
\newblock All alternating groups {$A_n$} with {$n\geq12$} have polytopes of
  rank {$\lfloor\frac{n-1}{2}\rfloor$}.
\newblock {\em SIAM J. Discrete Math.}, 26(2):482--498, 2012.

\bibitem{FLM12A}
Maria~Elisa Fernandes, Dimitri Leemans, and Mark Mixer.
\newblock Polytopes of high rank for the alternating groups.
\newblock {\em J. Combin. Theory Ser. A}, 119(1):42--56, 2012.

\bibitem{GAP}
{GAP} {\textendash} {G}roups, {A}lgorithms, and {P}rogramming, {V}ersion
  4.11.0.
\newblock \href {https://www.gap-system.org}
  {\texttt{https://www.gap-system.org}}, Feb 2020.

\bibitem{flag-bicolorings}
Hiroki Koike, Daniel Pellicer, Miguel Raggi, and Steve Wilson.
\newblock Flag bicolorings, pseudo-orientations, and double covers of maps.
\newblock {\em Electron. J. Combin.}, 24(1):Paper No. 1.3, 23, 2017.

\bibitem{leemans-mulpas-gluing}
Dimitri Leemans and Jessica Mulpas.
\newblock Two gluing methods for string {C}-group representations of the
  symmetric groups.
\newblock {\em Comb. Theory}, 6(1):Paper No. 13, 17, 2026.

\bibitem{ARP}
Peter McMullen and Egon Schulte.
\newblock {\em Abstract regular polytopes}, volume~92 of {\em Encyclopedia of
  Mathematics and its Applications}.
\newblock Cambridge University Press, Cambridge, 2002.

\bibitem{CPR}
Daniel Pellicer.
\newblock C{PR} graphs and regular polytopes.
\newblock {\em European J. Combin.}, 29(1):59--71, 2008.

\bibitem{parallelProduct}
Stephen~E. Wilson.
\newblock Parallel products in groups and maps.
\newblock {\em J. Algebra}, 167(3):539--546, 1994.

\end{thebibliography}

\end{document}